\documentclass[11pt]{amsart}
\usepackage{amssymb,amscd,graphics,verbatim,mathrsfs,latexsym,amsmath,amsthm,eucal}

\makeatletter
\@addtoreset{equation}{section}\makeatother

\newtheorem{theo}{Theorem}[section]
\newtheorem{lem}[theo]{Lemma}
\newtheorem{prop}[theo]{Proposition}
\newtheorem{cor}[theo]{Corollary}
\newtheorem{remark}[theo]{Remark}

\newcommand{\mc}{\mathcal}
\newcommand{\rr}{\mathbb{R}}
\newcommand{\nn}{\mathbb{N}}
\newcommand{\cc}{\mathbb{C}}
\newcommand{\hh}{\mathbb{H}}

\newcommand{\eps}{\epsilon}

\newcommand{\pl}{\partial}
\newcommand{\x}{\times}

\newcommand{\til}{\widetilde}
\newcommand{\bbar}{\overline}

\newcommand{\supp}{\textrm{supp}}
\newcommand{\cjd}{\rangle}
\newcommand{\cjg}{\langle}

\newcommand{\demi}{\frac{1}{2}}
\newcommand{\ndemi}{\frac{n}{2}}

\newcommand{\indic}{\operatorname{1\negthinspace l}}

\newcommand{\ZZ}{\mathbb{Z}}

\newcommand{\del}{\partial}

\newcommand{\la}{\lambda}

\newcommand{\calC}{{\mathcal C}}

\newcommand{\calF}{{\mathcal F}}

\newcommand{\calH}{{\mathcal H}}
\newcommand{\calI}{{\mathcal I}}

\newcommand{\calL}{{\mathcal L}}

\newcommand{\calO}{{\mathcal O}}

\newcommand{\calU}{{\mathcal U}}
\newcommand{\calV}{{\mathcal V}}

\title[Resolvent of the Laplacian]{Resolvent of the Laplacian \\ on geometrically finite hyperbolic manifolds}
\author{Colin Guillarmou}
\address{DMA, U.M.R. 8553 CNRS\\
Ecole Normale Sup\'erieure\\
45 rue d'Ulm\\ 
F 75230 Paris cedex 05 \\France}
\email{cguillar@dma.ens.fr}
\author{Rafe Mazzeo}
\address{Department of Mathematics\\
Stanford University\\
Stanford, CA 94305, USA.}
\email{mazzeo@math.stanford.edu.}

\begin{document}
%
\begin{abstract}
For geometrically finite hyperbolic manifolds $\Gamma\backslash\hh^{n+1}$, 
we prove the meromorphic extension of the resolvent of Laplacian, 
Poincar\'e series, Einsenstein series and scattering operator to the whole complex plane. 
We also deduce the asymptotics of lattice points of $\Gamma$ in large balls of $\hh^{n+1}$ in terms of 
the Hausdorff dimension of the limit set of $\Gamma$.
\end{abstract}
\maketitle

\section{Introduction}
Analysis of the Laplace operator on $(n+1)$-dimensional hyperbolic manifolds which satisfy a geometric finiteness
condition commenced in earnest in the early 1980's, inspired by numerous results in the finite volume setting, as well as 
some extensions of this, by Roelcke and Patterson, to infinite area geometrically finite surfaces.  The paper of Lax and Phillips 
\cite{LP}, see also \cite{LP2}, shows that the spectrum of the Laplacian on such spaces is equal to $[n^2/4,\infty)\cup S$ 
where $S\subset (0,n^2/4)$ is a finite set of $L^2$-eigenvalues, each of finite multiplicity, and also that the essential spectrum 
is entirely absolutely continuous. One important geometric motivation 
in their work was to deduce sharp asymptotics of the lattice point counting function for a geometrically finite group 
of isometries $\Gamma$ of $\hh^{n+1}$, under certain assumptions on the dimension of the limit set of $\Gamma$. 
This was followed by an extensive development by many authors concerning the special class of geometrically finite quotients which are convex cocompact, i.e.\ where $\Gamma$ has no parabolic elements. In particular, the second author and Melrose \cite{MM} proved
in the more general setting of  asymptotically hyperbolic metrics that the resolvent $R_X(s)=(\Delta_X-s(n-s))^{-1}$ has a meromorphic 
continuation with poles of finite rank in $s\in \cc\setminus \{\ndemi -\demi \nn\}$. Actually, that paper claimed erroneously that
the continuation is meromorphic in all of $\mathbb C$. This was rectified by the first author \cite{GuiDMJ}, who proved that the metric 
must have an even Taylor expansion at the boundary (in a suitable sense) in order for the resolvent to be meromorphic with 
finite rank poles also near $\ndemi-\demi\nn$. Very recently, Vasy \cite{Vas} found a new proof of this meromorphic extension
for asymptotically hyperbolic metrics which satisfy the same evenness condition; in particular, he obtains high frequency 
(semi-classical) estimates for $R_X(s)$ in the non-trapping case.  Guillop\'e and Zworski \cite{GZ2} gave a simpler proof of the 
main result of \cite{MM} assuming that the curvature is constant near infinity. Their technique is reviewed below. 
This extension of $R_X(s)$ was a main step in the development of scattering theory on these spaces, which is an important area in its 
own right, but also a means to prove various trace formul\ae, and a fundamental tool to analyze divisors of Selberg zeta function 
$Z_X(s)$ and lengths of closed geodesics on $X$. Analysis of these divisors for convex cocompact  hyperbolic manifolds was carried 
out in great detail by Patterson-Perry \cite{PP} and Bunke-Olbrich \cite{BO}; in this setting, a trace formula relating poles of the resolvent 
(called \emph{resonances}) and lengths of closed geodesics is also known \cite{Pe2,GN}, a sharp asymptotic 
with remainder for the counting function for lengths of closed geodesics is given in \cite{GN} (earlier versions without estimates
on the remainder appear in \cite{Mar,PaPol,GuDMJ,La,Pe}), and estimates on the distributions of resonances in $\cc$ are proved in \cite{GZ2,Pe2, Bo}.  

The two-dimensional case is special because the geometry is much simpler, and all of these results for geometrically finite surfaces
are essentially contained in the work of Guillop\'e-Zworski \cite{GZ} and Borthwick-Judge-Perry \cite{BJP}; the book by Borthwick 
\cite{Bo2} contains a unified treatment of this material.  In higher dimension, the geometry of nonmaximal rank cusps is
 more complicated, as we explain below, and this makes the analysis substantially more delicate.  The first work in
this generality was by the second author and Phillips \cite{MP}, where the spaces of $L^2$ harmonic differential forms were studied
and interpreted in topological terms. Subsequently, Froese-Hislop-Perry \cite{FHP} proved the existence of a meromorphic extension 
to $s\in \cc$ of the scattering operator and resolvent for geometrically finite hyperbolic $3$-dimensional manifolds, and recently
the first author \cite{GuiCPDE} extended this to higher dimensions when the cusps are `rational' (see below) and gave 
a bound on the counting function for resonances. The case left open is when there are nonmaximal rank cusps with irrational holonomy 
(this never happens in three dimensions). By contrast with these analytic approaches, Bunke and Olbrich \cite{BO1} developed representation 
theoretic methods to study the scattering operator in the general geometrically finite setting and proved its meromorphic extension  to $\cc$;
they do not study the resolvent. Their paper is a revision of an older treatment they gave of this subject, but contains a substantially new 
exposition; their approach is quite technical and may be inaccessible for those without the representation theory background. 
In any case, their techniques are of a completely different nature to ours and it is not simple to even compare the results. 

We consider here a general geometrically finite hyperbolic manifold $X:=\Gamma\backslash \hh^{n+1}$ and give a rather short proof 
of the meromorphic extension of the resolvent of the Laplacian $\Delta_X$ to $s\in\cc$, as well as applications 
to scattering theory and distributions of lattice points in $\hh^{n+1}$ in the spirit of Lax-Phillips. 
\begin{theo}\label{th1}
Let $X=\Gamma\backslash \hh^{n+1}$ be a geometrically finite hyperbolic manifold and $\Delta_X$ its Laplacian.
Then the resolvent $R_X(s):=(\Delta_X-s(n-s))^{-1}$, defined initially as a bounded operator on $L^2(X)$ for $\{{\rm Re}(s)>n/2, 
s(n-s)\notin S\}$, extends to a family of continuous mappings $\calC^\infty_0(X) \to \calC^\infty(X)$ which depends 
meromorphically on $s\in\cc$, with poles of finite rank. In addition, $R_X(s)$ is bounded on appropriate weighted Sobolev 
spaces (see Theorem \ref{mainth} for a precise statement). 
\end{theo}
We also describe fine mapping properties of $R_X(s)$ in ${\rm Re}(s)\geq n/2$  (i.e.\ on the continuous spectrum), which implies 
a limiting absorption principle in this setting. 

\subsection{An outline of the proof}
The proof relies as usual on a parametrix construction. It is enough to construct local parametrices
with this continuation property near every point of the conformal boundary at infinity (i.e. $\Gamma\backslash\Omega_\Gamma$
where $\Omega_\Gamma\subset S^n=\pl\hh^{n+1}$ is the domain of discontinuity of $\Gamma$); this is now standard at all but the `cusp points',
and hence one of the main steps is to prove this continuation when $\Gamma = \Gamma_p$ is a parabolic group of nonmaximal rank $k<n$, which
fixes a single point $p$ on $S^n=\pl\hh^{n+1}$. Such a group contains a maximal abelian subgroup of finite index, so we can reduce to the 
case where $\Gamma_p$ is abelian. Taking $p = \infty$ in the upper half-space model, we decompose the model cusp $X_c:=\Gamma_\infty
\backslash \hh^{n+1}$ as $X_c=(0,\infty)_x\x F$, where $F=\Gamma_\infty\backslash \rr^n$ is a flat bundle and $\Gamma_\infty$ acts as
a discrete group of Euclidean isometries on each horosphere $\{x={\rm const}\}\simeq \rr^n$. In this decomposition, the Laplacian has the form
\[\Delta_{X_c}=-(x\pl_x)^2+x^2\Delta_F+\frac{n^2}{4} \quad {\rm on }\,\, \, L^2\Big(\rr^+_x, \frac{dx}{x}; L^2(F)\Big)\]
where $\Delta_F$ is the flat Laplacian on the flat bundle $F$.  Next, by spectral decomposition we can reduce $\Delta_F$ to a parameter, and
hence regard $\Delta_{X_c}$ as a $2^{\mathrm{nd}}$ order Bessel-type ordinary differential operator on the half-line $(0,\infty)_x$, depending on 
this parameter. It is standard to write down an explicit formula for its resolvent kernel $R_{X_c}(s;x\sqrt{\Delta_F},x'\sqrt{\Delta_F})$ (using
homogeneous solutions of this ODE, which are modified Bessel functions). We define a functional calculus for $\Delta_F$ using a rather 
explicit spectral decomposition of the Laplacian $\Delta_F$ on the flat bundle; more precisely, we show that it decomposes as a 
countable direct sum $\Delta_F=\bigoplus_{I\in\mc{I}}\Delta_I$ of operators  
\[\Delta_I=-\pl_r^2-\frac{n-k-1}{r}\pl_r+\frac{m(m+n-k-2)}{r^2}+b_I^2 \quad {\rm on }\,\,\,L^2(\rr^+_r, r^{n-k-1}dr)\]
for some $b_I\geq 0$, where $I=(m,p,v)$ with $m\in\nn_0$, $p\in \nn$ and $v$ lies in the rank $k$ lattice $\Lambda^*$ dual 
to the translation part of $\Gamma_\infty$. Such an operator can be recognized as the flat Laplacian on $\rr^{n-k}$ acting 
on $m$-th sperical harmonics, but shifted with the constant $b_I^2$, its spectrum is $[b_I^2,\infty)$ and is absolutely continuous.
The number $b_I$ are quite explicit and behave quite differently if there is an element $\gamma\in\Gamma_\infty$ which is irrational 
in the sense that  no power of $\gamma$ is a pure translation on the horospheres, or equivalently, the flat bundle $F$ has holonomy 
representation in $\mathrm{O}(n-k)$ with infinite image. In the irrational case, the set $\{b_I\}_I$ accumulates at $0$. The resolvent 
kernel  $R_{X_c}(s,x\sqrt{\Delta_F},x'\sqrt{\Delta_F})$ involves complex powers $\Delta_F^{s}$ as the `parameter' $\Delta_F$ approaches $0$.  
The fact that the spectra of the $\Delta_I$ are not bounded away from zero creates fundamentally new complications, which are the root of the 
technical difficulties here, and is the reason that the spectral analysis in this general setting has not been treated analytically before. 
The difficulty is at low frequencies in $\Delta_F$, and a key observation is that the spectral measure $dE_I(t)$ of $\Delta_I-b_I^2$ vanishes 
approximately like $t^{2m}$ as $t \searrow 0$. Finally, we use a parametrix construction as in Guillop\'e-Zworski \cite{GZ2} to
use this resolvent on the model cusp $X_c$ to study the resolvent on an arbitrarily geometrically finite hyperbolic manifold. 
  
Our method is quite robust: combining the construction here with the parametrices constructed in \cite{MM}, it is possible to prove
the same result for asymptotically hyperbolic manifolds (with non-constant curvature) with certain neighbourhoods of infinity  
isometric to one of these constant curvature cusps $X_c$ (see Remark \ref{moregeneral}), and with asymptotic geometry at 
all other points at infinity `conformally compact'; by more classical methods still, we may also allow maximal rank cusps. 
In the interests of keeping the presentation simple and short, we omit any further discussion of these generalizations. The spectral 
decomposition for noncompact flat manifolds developed here does not seem to appear in the literature; our treatment of the
general case was inspired by Gilles Carron's explanation of the three-dimensional case. 

\subsection{Applications}
One corollary concerns the Poincar\'e series and lattice point counting function for the group $\Gamma$.
Recall that, given $\Gamma$, there exists a number $\delta = \delta(\Gamma) \in (0,n)$, called the Poincar\'e exponent of
$\Gamma$, such that for any $m,m'\in \hh^{n+1}$, the Poincar\'e series 
\[P_s(m,m'):=\sum_{\gamma\in\Gamma}e^{-sd(m,\gamma m')}\]
(where $d(\cdot, \cdot) = d_{\hh^{n+1}}(\cdot,\cdot)$ is distance in $\hh^{n+1}$) converges for ${\rm Re}(s)>\delta$; by a famous 
result of Patterson \cite{PatActa} and Sullivan \cite{SU}, this $\delta$ is precisely the Hausdorff dimension of the limit set 
$\Lambda_\Gamma$ of  $\Gamma$. 
\begin{cor}\label{cor1}
Let $X=\Gamma\backslash \hh^{n+1}$ be a geometrically finite hyperbolic manifold, $m,m'\in \hh^{n+1}$ and  $P_s(m,m')$ the 
corresponding Poincar\'e series. Then $P_s(m,m')$  extends meromorphically from $\{{\rm Re}(s)>\delta\}$ to $s\in\cc$ and 
there are constants $c,c'>0$ depending only on $\Gamma$ such that as $R \to \infty$, 
\begin{equation}\label{counting}
\sharp \{\gamma\in \Gamma; d(m,\gamma m')\leq R\}\sim c\, e^{\delta R}F(m)F(m').
\end{equation} 
Here $0 < F \in \calC^\infty(\hh^{n+1})$ is the $\Gamma$-automorphic function given in terms of the Poisson kernel $P(m,\zeta)$ 
and the Patterson-Sullivan probability measure $\mu_\delta$ supported on $\Lambda_\Gamma$ by $F(m)=c'\int_{S^n}P(m,\zeta)^\delta 
d\mu_{\delta}(\zeta)$, and $F$ satisfies ${\rm Res}_{s=\delta}R_X(s)=F\otimes F$;
\end{cor}
This result was known when $\delta>n/2$ by the work of Lax-Phillips \cite{LP} (with exponential error terms), and in the convex cocompact case
by Patterson \cite{PatArx} without any condition on $\delta$. Using Patterson-Sullivan theory, Roblin \cite{Rob} obtained
the asymptotics \eqref{counting}  of lattice points under weaker assumptions than ours, but he does not prove the continuation of the 
Poincar\'e series and his techniques are significantly different. In order to obtain error terms in \eqref{counting} when $\delta<n/2$, it
is necessary to prove that there is a strip in $\cc$ which is free of resonances, but currently this is only known in the convex cocompact 
case \cite{Naud,Sto}. 

In the final section, we define the Eisenstein series and scattering operator $S_X(s)$, which describe the asymptotic behaviour of 
generalized eigenfunctions of $\Delta_X$ near infinity, and prove their meromorphic continuation to $s\in \cc$. We also establish
several typical functional equations for these operators and prove that $S_X(s)$ is a pseudodifferential operator
acting on the manifold $B:=\Gamma\backslash \Omega_\Gamma$, where $\Omega_\Gamma := S^n \backslash \Lambda_\Gamma \subset S^n$ 
is the domain of discontinuity of $\Gamma$. Thus $B$ is the `boundary at infinity' of $X$, 
and is the natural `locus' of scattering. For convex cocompact quotients, $B$ is compact (and the fact that $S_X$ is pseudodifferential
is well known), but in the geometrically finite case, $B$ is noncompact with finitely many ends, each of which corresponds 
to a cusp of $X$ and is identified with (the end of) a flat vector bundle over a compact flat manifold. We show 
that $S_X(s)$ is equal to the sum of the complex power $\Delta_B^s$ of the Laplacian $\Delta_B$ and a residual term. 

We do not estimate the growth of the counting function of resonances here, but it is certainly possible to do this using our
construction; this will be carried out elsewhere.  It is likely that some Diophantine condition on the irrational elements of $\Gamma$
may be needed to show that this counting function has polynomial growth. All this should be a fundamental step towards proving a Selberg trace 
formula  and the extension of Selberg zeta function for general geometrically finite groups, and would also have applications to 
the study of the length distribution of geodesics. 

\subsection{Organization of the paper} We first review the geometry of geometrically finite hyperbolic quotients, following 
Bowditch \cite{Bow}, and then recall the parametrix construction of Guillop\'e-Zworski \cite{GZ2} for the resolvent when there are 
no cusps. In \S 4 we develop the spectral decomposition of the flat bundle $\Gamma_\infty \setminus \rr^n$, and then
use this in \S 5 to prove the meromorphic continuation of the resolvent on the model non-maximal rank cusp $X_c = \Gamma_\infty \backslash \hh^{n+1}$. \S 6 contains our main result, while in \S 7 we give the finer descriptions of the resolvent kernel close to the 
critical line ${\rm Re}(s)=n/2$ by using summation over the group and then prove the meromorphic continuation of 
Poincar\'e series and its application to lattice points counting in Corollary \ref{cor1}. In the final section, we consider the scattering operator 
and study the Eisenstein series (or Poisson operator).

\section{Geometry of geometrically finite hyperbolic manifolds}
We view hyperbolic $(n+1)$-space $\hh^{n+1}$ either as the half-space $\rr^+_x\x \rr^n_y$ or as the open unit ball 
$B^{n+1}\subset \rr^{n+1}$; its natural smooth compactification $\bbar{\hh}^{n+1}=\bbar{B}^{n+1}$ is obtained by gluing on the 
unit sphere $S^n\subset \rr^{n+1}$. Let $\Gamma$ be a discrete torsion-free group of isometries of $\hh^{n+1}$, 
$X:=\Gamma\backslash\hh^{n+1}$ is a smooth manifold. Note that the action of $\Gamma$ extends to $\bbar{\hh}^{n+1}$, but
any element always has fixed points on the boundary at infinity $S^n$.  If $m \in \hh^{n+1}$, then the set of accumulation points 
of the orbit $\Gamma \cdot m$ in $\bbar{\hh}^{n+1}$ is a closed subset $\Lambda_\Gamma \subset S^n$ called the \emph{limit set} of $\Gamma$.
Its complement $\Omega_\Gamma:=S^n\setminus \Lambda_\Gamma$ is called the \emph{domain of discontinuity}, and $\Gamma$ acts 
properly discontinuously in $\Omega_\Gamma$, with quotient $B = \Gamma \backslash \Omega_\Gamma$. 

We now specialize to the setting where $\Gamma$ is geometrically finite; good references for this include 
Bowditch \cite{Bow} and the monograph by Ratcliffe \cite[Chap. 12]{Rat}. 

Since $\Gamma\setminus\{\rm Id\}$ has no fixed points in $\hh^{n+1}$, any $\gamma \in \Gamma$ has either one or two fixed points on 
$S^n$; in the former case it is called \emph{parabolic}, and in the latter, \emph{hyperbolic} (or sometimes also
\emph{loxodromic}). In either case, the fixed point set of any $\gamma \in \Gamma$ lies in $\Lambda_\Gamma$.
If $p$ is the fixed point of a parabolic element, then the subgroup $\Gamma_p\subset \Gamma$ stabilizing $p$ contains 
only parabolic elements, and is called an \emph{elementary parabolic group}.  Conjugating by a suitable isometry, 
we may assume that $p = \infty$ in the upper half-space model. Then, viewing $\hh^{n+1}$ as the half plane $\rr^+\x \rr^n$
with $x>0$ the vertical variable, $\Gamma_\infty$ acts isometrically on each horosphere $E_a:=\{x=a\}\simeq \rr^n$ with
the induced Euclidean metric. It was proved by Bieberbach, see \cite[Sec. 2.2]{Bow} and \cite[Sec. 2]{MP}, 
that there is a maximal normal abelian subgroup $\Gamma'_\infty\subset \Gamma_\infty$ of finite index and an affine suspace 
$Z\subset E_1$ of dimension $k$, invariant under $\Gamma_\infty$, such that $\Gamma'_\infty$ acts as a group of translations of
rank $k$ on $Z$, so that the quotient $T':=\Gamma'_\infty \backslash Z$ is a $k$-dimensional flat torus. If $E_1=Z\x Y$ is
an orthogonal decomposition, with $Y\simeq \rr^{n-k}$ and associated coordinates $(z,y)$, then each $\gamma\in 
\Gamma_\infty$ acts by 
\[ 
\gamma (x,y,z)= (x, A_\gamma y, R_\gamma z+b_{\gamma}), \quad b_\gamma\in \rr^k, \quad R_\gamma\in \mathrm{O}(n-k),
A_\gamma \in \mathrm{O}(k),
\]
where for each $\gamma$, $R_\gamma^m={\rm Id}$ for some $m \in\nn$, with $m=1$ if $\gamma\in \Gamma'_\infty$. If there 
exists $m\in\nn$ such that $\gamma^m(x,y,z)=(x,y,z+c_\gamma)$ for some $c_\gamma\in\rr^k$, then $\gamma$ is called \emph{rational}, 
which is equivalent to saying that $R_\gamma^m={\rm Id}$ and $A_\gamma^m={\rm Id}$ for some $m\in\nn$. Otherwise 
$\gamma$ is called \emph{irrational}. The quotients $\Gamma_\infty'\backslash \hh^{n+1}$ and $\Gamma_\infty\backslash\hh^{n+1}$ 
are both of the form $\rr^+_x\x F$ and $\rr^+_x\x F'$ for some flat bundles $F\to T$ and $F'\to T'$, where the bases $T$ and $T'$
are compact flat manifolds; here $F=\Gamma_\infty\backslash E_1$, $F'=\Gamma'_\infty\backslash E_1$ and $T=\Gamma_\infty\backslash Z$.
The hyperbolic metric on $\hh^{n+1}$ descends to a hyperbolic metric $g_X=(dx^2+g_F)/x^2$ where $g_F$ is a flat metric
on the bundle $F$ induced from the restriction of the hyperbolic metric to the horosphere $\{x=1\}$. 

Using the splitting $\hh^{n+1}=\rr^+\x Z\x Y$, define
\[
C_\infty(R):= \{(x,z,y)\in [0,\infty)\x Z\x Y ; x^2+|y|^2\geq R\} \subset \bbar{\hh}^{n+1};
\]
this is invariant under $\Gamma_\infty$ and is hyperbolically convex. It is called a \emph{standard parabolic region} for $\Gamma_\infty$. Similarly,
define $ C_p(R)$ for any other parabolic fixed point $p$. 

For any parabolic fixed point $p$, there exists an $R > 0$ so that the parabolic region $C_p(R)$ satisfies
\[
C_p(R)\subset\hh^{n+1}\cup \Omega_\Gamma, \quad \mbox{and}\quad \gamma C_p(R)\cap C_p(R)=\emptyset
\ \mbox{for all}\ \gamma\in\Gamma\setminus \Gamma_p.
\]
These conditions imply that $C_p(R)$ descends to a set $\mc{C}_p:=\Gamma\backslash (\cup_{\gamma\in\Gamma}\gamma(C_p(R)))$.
This is contained in $\Gamma \backslash(\hh^{n+1}\cup\Omega_\Gamma)$ and has interior isometric to the interior of $\Gamma_\infty
\backslash C_P(R)$. The set $\mc{C}_p$ is called a \emph{standard cusp region} associated to (the orbit of) $p$. The \emph{rank of the cusp}
is the rank of $\Gamma'_\infty$. 

A \emph{geometrically finite hyperbolic quotient} is a quotient $\Gamma\backslash \hh^{n+1}$ by a discrete group $\Gamma$
such that $\Gamma\backslash (\hh^{n+1}\cup\Omega_\Gamma)$ has a decomposition into the union of a compact set $K$ and a finite number
of standard cusp regions. This is more general than requiring that there exist a convex finite-sided fundamental domain of $\Gamma$,
although these conditions are equivalent when $n\leq 3$ or if all $\gamma\in\Gamma$ are rational (see Prop.5.6 and 5.7 in \cite{Bow}). 
When $\Gamma$ has no elliptic or parabolic elements, then both $\Gamma$ and the quotient $\Gamma\backslash \hh^{n+1}$
are called  \emph{convex co-compact}, and the quotient manifold has no cusp then.

In general, $X=\Gamma\backslash \hh^{n+1}$ with hyperbolic metric $g$ is a complete noncompact hyperbolic manifold with $n_c$ cusps, where $n_c$ is the number of $\Gamma$-orbits of fixed points of the parabolic elements of $\Gamma$. 
The conjugacy class of the parabolic subgroup fixing the $j^{\mathrm{th}}$ cusp point is denoted $\Gamma_j$. By geometric finiteness,
$X$ has finitely many ends and there exist a covering of the ends of $X$ of the form
$\{\calU_j^r\}_{j\in J^r} \cup \{\calU^c_j\}_{j\in J^c}$, $|J^c|=n_c$, so that $X$ minus the union of all these sets is compact, 
each $\calU^r_j$ is isometric to a half-ball in $\hh^{n+1}$, and each $\calU^c_j$ is isometric to a cusp region
$\Gamma_j \backslash C_\infty(R_j)$. 
We also assume that the $\calU_j^c$ are disjoint; they are called \emph{cusp neighbourhoods}; the $\calU^r_j$ are 
called \emph{regular neighbourhoods}.  

We also consider the smooth manifold with boundary $\bbar{X}:=\Gamma\backslash(\hh^{n+1}\cup\Omega_\Gamma)$,
which is noncompact, with finitely many ends, if $n_c > 0$. There is a compactification of $\bbar{X}$ as a smooth compact manifold 
with corners, see \cite{MP}, but we do not need this here. The boundary $\pl\bbar{X}=\Gamma\backslash\Omega_\Gamma$ is also a
noncompact manifold with $n_c$ ends; if we denote by $x$ a boundary defining function of $\pl\bbar{X}$ in $\bbar{X}$ 
which extends the function $x$ defined in each $\calU^c_j$ (as transferred to $X$ via the appropriate isometry), 
then $g_{\pl\bbar{X}}:=(x^2g_X)|_{T\pl\bbar{X}}$ defines a complete metric on $\pl\bbar{X}$ which is flat outside a compact set.
Indeed, writing the hyperbolic metric $g$ in $\calU_j^c$ as $x^{-2}(dx^2+g_{F_j})$, then near the cusp point 
$p_j$, $g_{\pl\bbar{X}}$ is the metric naturally induced from $g_{F_j}$ at $x=0$. There is an associated volume form $dv_{\pl\bbar{X}}$.
(Note, however, that this metric is only well-defined up to a positive smooth multiple away from the ends, and up to a constant
multiple in the ends; this does not make any difference for us.) 

\medskip

We conclude this section with the description of several different function spaces which appear frequently below.
 
The first is the standard space $\calC_0^\infty(X)$ of smooth functions compactly supported in the interior $X$
of $\bbar{X}$. We also consider $\calC_0^\infty(\bbar{X})$, which by definition consists of smooth functions on $\bbar{X}$ with 
compact support in the (non-compact) manifold $\bbar{X}$, so in particular the intersection of the support of such a $\phi$ with 
any $\calU^c_j$ lies in $\{x^2+|y|^2\leq R\}$ for some $R> R_j > 0$.  The functions in $\calC_0^\infty(\bbar{X})$ which vanish to infinite order 
at the boundary $\pl\bbar{X}$ constitute the space $\dot{\calC}_0^\infty(\bbar{X})$. 

We next define the $L^2$-based Sobolev spaces on $\pl\bbar{X}$ with respect to the metric $g_{\pl\bbar{X}}$:
\[ 
H^M(\pl\bbar{X}):=\{ f\in L^2(\pl\bbar{X}, dv_{\pl\bbar{X}}); \nabla_{\pl\bbar{X}}^{\ell}f\in L^2(\pl\bbar{X},dv_{\pl\bbar{X}}), \ \forall \ell\leq M\},
\]
and their intersection $H^\infty(\pl\bbar{X})=\cap_{M\geq 0}H^M(\pl\bbar{X})$. We also define
\[
H^{\infty}(X):=\{ f\in \calC^\infty(\bbar{X}); f|_{\calU^c_j}\in \calC_b^\infty([0,\infty)_x, H^{\infty}(F_j)) \textrm{ for all }j\in J^c \},
\] 
where we regard each $\calU^c_j$ as lying in $[0,\infty)_x \x F_j\simeq \Gamma_j\backslash\hh^{n+1}$,
and where for any Frechet space $E$, $\calC^\infty_b([0,\infty),E)$ denotes the set of smooth $E$-valued functions $f$ 
with all derivatives $\pl_x^\ell f$ bounded uniformly in $x$ with respect to each semi-norm of $E$. In particular, restrictions of 
functions in this space to $\pl\bbar{X}$ belong to $H^\infty(\pl\bbar{X})$. Finally we define 
\[ 
\dot{H}^\infty(X):=\{ f\in H^\infty(X); f=\calO(x^\infty) \textrm{ as  }x\to 0 \textrm{ and }f=\calO(x^{-\infty}) \textrm{ as }x\to \infty\}.
\]

\section{The resolvent when $\Gamma\backslash\hh^{n+1}$ has no cusps}
Since our main construction below uses the arguments of Guillop\'e-Zworski \cite{GZ2} (which they developed for the case 
with no cusps), we recall their method now. The construction is based on two things: the resolvent on $\hh^{n+1}$ and the indicial equation
for the Laplacian near the boundary at infinity. 
\subsection{Model resolvent}
The  kernel of the resolvent $R_{\hh^{n+1}}(s)=(\Delta_{\hh^{n+1}}-s(n-s))^{-1}$ of 
$\Delta_{\hh^{n+1}}$ is a function of the distance $d(\cdot,\cdot)$ on $\hh^{n+1}$ given by \cite[Sec 2]{GZ2}
\[\begin{split}
R_{\hh^{n+1}}(s;m,m')=&\frac{\pi^{-\ndemi}2^{-2s-1}\Gamma(s)}{\Gamma(s-\ndemi+1)}\cosh^{-2s}\Big(\frac{d(m,m')}{2}\Big) \\
&\x F\Big(s,s-\ndemi+\demi,2s-n+1;\cosh^{-2}\Big(\frac{d(m,m')}{2}\Big)\Big)
\end{split}\] 
where $F(a,b,c;u)=1+
\tfrac{a . b}{1. c}u+\tfrac{a(a+1).b(b+1)}{1.2.c(c+1)}u^2+\dots$ is the hypergeometric function. This 
is holomorphic in $s\in \cc$ when $n$ is even and meromorphic with poles at $s\in -\nn_0$ if $n$ is odd, 
and the residues are finite rank operators \cite[Lemma 2.2]{GZ2}.
There is an equivalent formula in terms of $\tau(m,m'):=1/\cosh(d(m,m'))$, see \cite[Lemma 2.1]{GZ2}:
\begin{equation}\label{formulemodel}
 R_{\hh^{n+1}}(s;m,m')= \frac{\tau^{s}(m,m')}{\pi^{\ndemi}2^{s+1}}\sum^\infty_{j=0}
\frac{\Gamma(s+2j)}{\Gamma(s-\ndemi+1+j)\Gamma(j+1)}(2\tau(m,m'))^{2j} ; 
\end{equation}
in the half-space model $\rr^+_x\x \rr^n_y$, $\tau(x,y,x',y')=\frac{2xx'}{|y-y'|^2+x^2+(x')^2}$.
For all $\eps>0$, this series converges uniformly  in $\tau\leq 1-\eps$. Define 
\begin{equation}\label{halfball}
B:= \{(x,y)\in (0,1)\x \rr^n; x^2+|y|^2<1\}, \quad \bbar{B}:= \{(x,y)\in [0,1)\x \rr^n; x^2+|y|^2<1\},
\end{equation}
viewed as subsets of $\hh^{n+1}$, and its compactification $\bbar{\hh}^{n+1}$. 
Let $\chi\in \calC^\infty(\bbar{B})$, then by \eqref{formulemodel}
\begin{equation}\label{outsidediag}
\chi(m) R_{\hh^{n+1}}(s;m,m')\chi(m')\in (xx')^s\calC^\infty(\bbar{B}\x\bbar{B}\setminus {\rm diag})
\end{equation}
where ${\rm diag}$ denotes the diagonal. 

\begin{lem}\label{composition}
Let $\chi\in \calC^\infty(\bbar{B})$ with compact support in $\bbar{B}$, let $s\in\cc$ and let
$K_s$ be an operator in $B$ with Schwartz kernel in $x^{N}{x'}^s\calC^\infty(\bbar{B}\x\bbar{B})$ 
for all $N\in\nn$, and with compact support in $\bbar{B}\x\bbar{B}$. 
Then $\chi R_{\hh^{n+1}}(s)\chi K_s$ is an operator in $B$ with Schwartz kernel in
$(xx')^s\calC^\infty(\bbar{B}\x\bbar{B})$.
\end{lem}
The proof follows by combining the composition result \cite[Theo. 3.18]{MaCPDE} and the description in 
\cite[Prop. 6.2]{MM} of the kernel $R_{\hh^{n+1}}(s;m,m')$ near the diagonal of the boundary $S^n=\pl\bbar{\hh}^{n+1}$.

\subsection{Indicial equation for the Laplacian}\label{secindicial}
The Laplacian in the half-space model $\rr^+_x\x \rr^n_y$ of $\hh^{n+1}$ is given by 
\[\Delta_{\hh^{n+1}}=-(x\pl_x)^2+nx\pl_x+x^2\Delta_y; \]
from this one calculates that if $F\in \calC^\infty([0,1]\x\rr^n)$, $\ell\in\nn_0$, then
\begin{equation} \label{indicialeq}
(\Delta_{\hh^{n+1}}-s(n-s))\frac{x^{s+2\ell}F(x^2,y)}{2\ell(n-2s-2\ell)}=x^{s+2\ell}(F(x^2,y)+
x^2H_{s,\ell}(x^2,y)),
\end{equation} 
where $H_{s,\ell} \in \calC^\infty([0,1]\x \rr^n)$ depends meromorphically on $s$ and has a single (simple) pole at $n/2-\ell$. 
Using \eqref{indicialeq} inductively, we see that for each $F\in \calC^\infty([0,1]\x\rr^n)$ and $N\in\nn_0$, 
there exists $G_{s,N},L_{s,N}\in \calC^\infty([0,1]\x\rr^n)$ 
with $\tfrac{\Gamma(s-n/2+1+N)}{\Gamma(s-n/2+1)}G_{s,N}$ 
and $\tfrac{\Gamma(s-n/2+1+N)}{\Gamma(s-n/2+1)}L_{s,N}$ holomorphic in $s$, and such that
\[(\Delta_{\hh^{n+1}}-s(n-s))x^{s+2}G_{s,N}(x^2,y)=x^{s+2}F(x^2,y)+x^{s+2+2N}L_{s,N}(x^2,y).\]
Using Borel summation, there exist $G_{s},L_{s}\in \calC^\infty([0,1]\x\rr^n)$ with $G_s/\Gamma(s-n/2+1)$ 
and $L_s/\Gamma(s-n/2+1)$ holomorphic in $s$, such that
\[(\Delta_{\hh^{n+1}}-s(n-s))x^{s+2}G_{s}(x^2,y)=x^{s+2}F(x^2,y)+L_{s}(x^2,y)\]
and $L_s(x^2,y)=\calO(x^\infty)$ at $x=0$.

\subsection{The parametrix construction of Guillop\'e-Zworski}\label{Guillopezworski}
Now, let $X=\Gamma\backslash \hh^{n+1}$ be a convex cocompact hyperbolic manifold. The smooth compactification 
$\bbar{X}=\Gamma\backslash(\hh^{n+1}\cup \Omega_\Gamma)$ is a compact manifold with boundary. Choose a smooth 
boundary defining function $\rho$ for $\pl\bbar{X}$ in $\bbar{X}$. There is a finite open cover $\{\calU_j\}$ of $X$ such that 
$X\setminus \cup_j\calU_j$ is compact in $X$ and each $\calU_j$ is identified by an isometry $\iota_j$ to the half-ball 
$B\subset \hh^{n+1}$ with hyperbolic metric $x^{-2}(dx^2+|dy|^2)$. The sets $\calU_j$ cover a neighbourhood of the
boundary at infinity $\pl\bbar{X}$ and $\iota_j^*x/\rho$ in $\calU_j$ is a smooth strictly positive function. 
Using pull-back by $\iota_j$, we can systematically identify operators on $\calU_j$ with their counterparts on $B$. In particular, denote
by $R_j(s)$ the operator obtained in this way from the restriction to $B$ of the meromorphically extended resolvent $R_{\hh^{n+1}}(s;m,m')$.
Fix $\chi_j,\hat{\chi}_j\in \calC^\infty(\bbar{X})$ with $\chi:=\sum_{j=1}^\ell \chi_j$ equal to $1$ near $\pl\bbar{X}$,
$\hat{\chi}_j$ supported in $\calU_j$ and $\hat{\chi}_j=1$ on the support of $\chi_j$. We also assume that 
each $\hat{\chi}_j$ is smooth as a function of $x^2$ in $\bbar{B}$.  

Define the initial parametrix $Q_0(s):=\sum_{j=1}^\ell \hat{\chi}_jR_j(s)\chi_j$. This satisfies
\[
(\Delta_{X}-s(n-s))Q_0(s)=\chi + K_0(s), \quad K_0(s):=\sum_{j=1}^\ell [\Delta_{X},\hat{\chi}_j]R_j(s)\chi_j.
\]
The expression \eqref{formulemodel} can be used to deduce that, as a Schwartz kernel on $\calU_j\x \calU_j$ (identified with $B\x B$),  
\[
[\Delta_X,\hat{\chi}_j]R_j(s)\chi_j (x,y,x',y')=x^{s+2}{x'}^s G(s;x,y,x',y')
\] 
for some function $G(s;m,m')$ which is smooth in $(m,m')\in B^2$ and meromorphic in $s$ with 
poles of finite rank at $-\nn_0$ if $n$ is odd and no poles if $n$ is even. Moreover, $G(s;x,y,x',y')$ is a smooth function of 
$(x^2,y,x',y')\in ([0,1]\x\rr^n)^2$. 

As explained in \S \ref{secindicial}, this `indicial equation' allows us to inductively solve away terms in the Taylor expansion 
in $x$ of the Schwartz kernel of $K_0(s;x,y,x',y')$ at $x=0$, viewing the right variable $(x',y')$ as a parameter.
In other words, for any $N > 0$ we can construct a kernel $Q_{N,j}(s)$ such that $(xx')^{-s}Q_{N,j}(s)$ is a smooth function of 
$(x^2,y,{x'},y')$ down to $x=0$ and $x'=0$, and 
\[
(\Delta_{X}-s(n-s))Q_{N,j}(s)=\chi_j+K_{N,j}(s)
\] 
for some $K_{N,j}(s)\in x^{s+2N}{x'}^s \calC^\infty(\bbar{\calU}_j\x \bbar{\calU}_j)$ (here $\bbar{\calU}_j$ is the closure in the compactification $\bbar{X}$), 
which are therefore Schwartz kernels of compact operators on $\rho^NL^2(X)$ in $\{{\rm Re}(s)>n/2-N\}$. 
The Taylor expansion at $x=0$ of $G(s;x,y,x',y')$ is essentially contained in the series expansion \eqref{formulemodel}, and all coefficients are meromorphic with only simple poles at $-\nn_0$, with corresponding residues some finite rank operators.  Then by \S \ref{secindicial}, the poles 
of $Q_{N,j}(s),K_{N,j}(s)$ are necessarily contained in $(n/2-\nn)\cup -\nn_0$, but a priori not necessarily of finite rank. 
A careful look at the coefficients of the expansion \eqref{formulemodel}, and in particular the $1/\Gamma(s-n/2+1+j)$ factor 
vanishing to first order at $s\in n/2-1-j-\nn$, shows that the poles occur only at $-\nn_0$, are simple, and the residues are finite rank. 
We omit details and refer the interested reader to \cite[Prop 3.1]{GZ2} for these facts. 
Finally, choosing $\eta\in \calC_0^\infty(X)$ so that $\eta(1-\chi)=1-\chi$, fix $s_0\gg n/2$ large 
and set $Q_N(s):=\sum_{j=1}^\ell Q_{N,j}(s)+\eta R_X(s_0)(1-\chi)$, then
\[\begin{gathered}
(\Delta_{X}-s(n-s))Q_{N}(s)=1+K_{N}(s), \quad \mbox{where} \\  
K_N(s):=\sum_{j=1}^\ell K_{N,j}(s)+[\Delta_X,\eta]R_X(s_0)(1-\chi) +(s_0(n-s_0)-s(n-s))\eta R_X(s_0)(1-\chi)
\end{gathered}\]
Then $K_N(s)\in \rho^{s+2N}\rho'^s\calC^\infty(\bbar{X}\x\bbar{X})$ is the kernel of a compact operator on $\rho^NL^2(X)$
if ${\rm Re}(s)>n/2-N$. It is shown in \cite[Prop 3.1]{GZ2} that for each $N\in \nn$, if $s_0$ is large enough and choosing the 
support of $\chi$ sufficiently close to $\pl\bbar{X}$, then we have $||K_N(s_0)||_{\rho^NL^2\to\rho^NL^2}<1/2$. The operator $Q_N(s)$ 
has Schwartz kernel constructed from the model resolvent $R_{\hh^{n+1}}(s)$ and terms 
in $(\rho\rho')^s \calC^\infty(\bbar{X}\x\bbar{X})$, it is straightforward to check\footnote{This is claimed without proof in \cite[Prop 3.1]{GZ2}, but 
is an easy exercise. It also follows from combining \cite[Prop 6.2]{MM} with \cite[Prop 3.20]{MaCPDE}} that this is a bounded operator
mapping $\rho^NL^2(X)$ to its dual $\rho^{-N}L^2(X)$; moreover, by the discussion above, it is meromorphic with poles of finite rank. 
Applying the Fredholm analytic theorem, we deduce that $(1+K_N(s))^{-1}$ exists as a meromorphic family of bounded operator on $\rho^NL^2(X)$ in $\{{\rm Re}(s)>n/2-N\}$, and thus $Q_N(s)(1+K_N(s))^{-1}$ meromorphically extends 
$R_X(s)$ to $\{{\rm Re}(s)>n/2-N\}$. Since this can be done for all $N\in\nn$, the proof is complete.

\section{Spectral decomposition when $\Gamma$ is an elementary parabolic group}
We first tackle the spectral analysis of the Laplacian when $\Gamma$ is a discrete elementary parabolic subgroup.
As before, and using the notation introduced in \S 2, assume that the parabolic fixed point is $\infty$ in the upper
half-space model. We use a type of Fourier decomposition for functions on the flat bundle $F = \Gamma_\infty \backslash \rr^n$.
This proceeds in two stages: we first obtain a discrete Fourier decomposition of functions on the compact spherical normal bundle 
$SF$, and then couple this with a continuous Fourier-Bessel type decomposition for functions in the radial variable
on the fibres of $F$; together these reduce $\Delta_F$ to a family of multiplication operators. 

Recall from \S 2 that we have a maximal abelian normal subgroup $\Gamma'_\infty\subset \Gamma_\infty$ of finite index and an affine 
$k$-dimensional subspace $Z\subset E_1$ on which $\Gamma'_\infty$ acts by translations, 
where $k$ is the rank of the cusp, and an orthogonal complement $Y\simeq \rr^{n-k}$  
in $E_1$. Since $[\Gamma_\infty:\Gamma'_\infty]$ is finite, it suffices to prove the meromorphic continuation of the resolvent 
on $\Gamma'_\infty\backslash\hh^{n+1}$; the resolvent on the  quotient $\Gamma_\infty \backslash \hh^{n+1}$ is then
a finite sum of translates of the resolvent on $\Gamma_\infty' \backslash \hh^{n+1}$. To simplify exposition we 
assume that $\Gamma_\infty =\Gamma'_\infty$ and that $X_\infty$ is orientable.  
Thus $\Gamma_\infty$ is freely generated by $k$ elements $\gamma_1,\dots,\gamma_k$ which 
in terms of the decomposition $\hh^{n+1}=\rr^+_x\x E_1$, $E_1=Y\oplus Z$, act by 
\[
\gamma_\ell(x,y,z) =(x, A_\ell y,z+v_\ell), \quad v_\ell\in\rr^k, A_\ell\in \mathrm{SO}(n-k).
\]
In other words, $\gamma_\ell$ is identified with the pair $(v_\ell,A_\ell)$. Since the $A_\ell$ are orthogonal, the Euclidean metric 
on $Y$ descends naturally to a flat metric $g_F$ on the fibres of $F:=\Gamma_\infty\backslash E_1$; in particular, the unit sphere bundle $SF$ 
is well-defined.

Since $\Gamma_\infty$ is abelian, $\{A_1, \ldots, A_k\}$ is a commuting set of orthogonal matrices, and so there is 
an orthogonal decomposition $Y\simeq \rr^{n-k} \cong V_0 \oplus \ldots \oplus V_s$, with $\dim V_0 = r$ and $\dim V_j 
= 2$, $j \geq 1$, such that $A_\ell$ acts trivially on $V_0$ and is a rotation by angle $\theta_{j\ell}$ on 
$V_j$ for every $j, \ell$. In other words, this decomposition puts each $A_j$ into block form with
each block either the identity or a rotation on each summand $V_j$.  

Altogether, we have now described $F$ as the total space of a vector bundle
$\calV$ over a compact $k$-dimensional torus $T$, where $\calV = \calV_0 \oplus \calV_1 \oplus 
\ldots \oplus \calV_s$; here $\calV_0$ is a trivial bundle of rank $r$ and all the other $\calV_j$ 
have rank $2$.  A function $f$ on $F$ is identified with a function $f(z,y)$ on $Z \oplus Y$
which satisfies 
\[
f(z + v_\ell, y) = f(z, A_\ell^{-1}y) = f(z, y_0, e^{-i\theta_{1\ell}} y_1, \ldots, e^{-i\theta_{s\ell}} y_s), \qquad
\ell= 1, \ldots, k, 
\]
where $(y_0, y_1, \ldots, y_s)$ are the components of $y$ with respect to the splitting $V_0 \oplus V_1 \oplus 
\ldots \oplus V_s$, and each $V_j$ is identified with $\cc$ when $j>0$.  We can describe functions $f$ on $SF$ in exactly the same way.

\subsection{Fourier decomposition on SF}\label{fourierdec}
The flat Laplacian $\Delta_{Y}$ on each fibre $F_z \simeq Y$ defines an operator on the total space of $F$ 
which acts fibrewise; its angular part is the Laplace operator on $S^{n-k-1}$. Let
\begin{equation}\label{L^2sphere}
L^2(S^{n-k-1}) = \bigoplus_{m=0}^\infty  H_m
\end{equation}
be the usual irreducible decomposition for the action of $\mbox{SO}\,(n-k)$, so $H_m$ is the space of spherical 
harmonics of degree $m$, 
\[
H_m = \ker \left(\Delta_{S^{n-k-1}} - m (m+n-k-2) \right);
\]
we also set $\dim H_m = \mu_m$.  

The vector spaces $H_m$ on each fibre of $SF$ fit together to form a flat 
vector bundle $\calH_m \longrightarrow T$. A section $\sigma$ of $\calH_m$ is a function $f(z,y)$ on $SF$,
identified as above with a function on $Z \times S^{n-k-1}$ satisfying $f(z+v_j,\omega) = f(z, A_j^{-1}\omega)$
such that for each $z \in T$, $\omega \mapsto f(z,\omega)\in H_m$.  (This makes sense since the action of 
$\mbox{SO}(n-k)$ preserves each $H_m$.) Now identify $V_j \cong \cc$, $j = 1, \ldots, s$, so that $A_j$ lies in the 
compact abelian subgroup $K = \times_{j=1}^s \mbox{U}(1)$, which acts by the identity on the first $r$ components 
in $\rr^{n-k}$.  The restriction of the irreducible representation of $\mbox{SO}(n-k)$ on $H_m$ to $K$ is a direct sum 
of one-dimensional irreducible representations:
\begin{equation}
H_m = L_1^{(m)} \oplus \cdots \oplus L_{\mu_m}^{(m)},
\label{eq:decompHm}
\end{equation} 
where $\theta = (\theta_1, \ldots, \theta_s) \leftrightarrow (e^{i\theta_1}, \ldots , e^{i\theta_{s}}) \in K$ acts 
on $L_p^{(m)}$ by $\exp (i \alpha_{mp}(\theta))$, $p = 1, \ldots, \mu_m$. It can be shown that
\[
\alpha_{mp}(\theta) = \sum_{\ell = 1}^s c_{mp\ell} \theta_\ell
\]
where the structure constants $c_{mp\ell}$ are all integers determined by the specific representation $L_p^{(m)}$; the 
precise formul\ae\ for them are complicated to state and in any case not important here.  

Because the $A_j$ lie in $K$, each $L^{(m)}_p$ determines a flat complex line bundle $\calL^{(m)}_{p}$ over $T$. 
By construction, a section $f$ of $\calL^{(m)}_p$ corresponds to a function $f(z,\omega)$ such that
\[
f(z + v_j, \omega) = e^{i \alpha_{mpj}} f(z, \omega), 
\]
where $\alpha_{mpj}$ are defined as
\begin{equation}
\alpha_{mpj} = \alpha_{mp} (\theta_{1 j}, \ldots, \theta_{s j}) 
\label{eq:holang}
\end{equation}
in terms of the holonomy angles $\theta_{\ell j}$ of $F$.

The orthogonal projection $\Pi_m: L^2(S^{n-k-1}) \longrightarrow H_m$ induces a map, which we still call $\Pi_m$, on 
$L^2(SF)$. Associated to (\ref{eq:decompHm}) are orthogonal subprojectors $\Pi_{mp}$, $p = 1, \ldots, \mu_m$, so 
that $\Pi_m = \bigoplus_p \Pi_{mp}$.  We often write $f_{mp} = \Pi_{mp} f$. 

If $f \in L^2(F)$, then using polar coordinates on each fibre, we write
\[
f(z,y) = f(z,r\omega) = \sum_{m=0}^\infty \sum_{p=1}^{\mu_m} f_{mp}(z,r,\omega).
\]
Let $\{v_1^*, \ldots, v_k^*\}$ be the basis for the lattice $\Lambda^*$ dual to $\Lambda := 
\{\sum a_j v_j: a_j \in \ZZ \}$, i.e.\  $\langle v_i, v_j^*\rangle = \delta_{ij}$
for all $i,j$, and define 
\[
f_{mp}(z,r,\omega) = e^{2 \pi i \langle z, A_{mp}\rangle} f_{mp}^{\#}(z,r,\omega), \qquad \mbox{where} \qquad
2\pi A_{mp} = \sum_{j=1}^k \alpha_{mpj} v_j^*.
\]
Then each $f^{\#}_{mp}$ is simply periodic,
\[
f^{\#}_{mp}(z+ v_j, r,\omega) = f^{\#}_{mp}(z,r,\omega), \quad j = 1, \ldots, k,
\]
hence standard Fourier series on $T$ gives the decomposition
\begin{equation}
f_{mp}(z,r,\omega) = \frac{1}{(2\pi)^k} \sum_{v^* \in \Lambda^*} \hat{f}_{mpv^*}(r,\omega) e^{2\pi i \langle z,  v^* + A_{mp}\rangle },
\end{equation}
where the $\hat{f}_{mpv^*}$ are the Fourier coefficients of $f^{\#}_{mpv^*}$. Clearly
\[
\int_{F} |f|^2\, dV = \sum_{m=0}^\infty \sum_{p = 1}^{\mu_m} \sum_{v^* \in \Lambda^*} \int_0^\infty \int_{S^{n-k-1}}
|\hat{f}_{mpv^*}(r,\omega)|^2\, r^{n-k-1}\, dr d\omega.
\]
For simplicity below, we write $I = (m,p,v^*)$, so $I$ ranges over the subset 
\[\calI = \{(m,p,v^*) \in
{\mathbb N} \times {\mathbb N} \times \Lambda^*: 1 \leq p \leq \mu_m\},\] 
and denote by
$\Pi_I$ the corresponding orthogonal projector on $L^2(SF)$ and $\phi_I(z,\omega)$ the associated 
eigenfunction.  We also simply write $f_I$ instead of $\hat{f}_I$. 

The Laplacian $\Delta_F$ is induced from the standard Laplacian on $T \times Y$ and
has the polar coordinate representation
\[
\Delta_F  = - \del_r^2 - \frac{n-k-1}{r}\del_r + \frac{1}{r^2} \Delta_{S^{n-k-1}} + \Delta_{T}.
\]
For each $I = (m,p,v^*)$ we have
\begin{equation}
(\Delta_F f)_I = \left(-\del_r^2 - \frac{n-k-1}{r}\del_r + \frac{m(m+n-k-2)}{r^2} + b_I^2\right) f_I,
\label{eq:redlapF}
\end{equation}
where
\begin{equation}
b_I = 2\pi |A_{mp} + v^*|. 
\label{eq:defCI}
\end{equation}
In the following we let $\Delta_I$ denote the operator on the right in (\ref{eq:redlapF}) acting on the
$I^{\mathrm{th}}$ component. 

\subsection{The radial Fourier-Bessel decomposition}
Using the spectral decomposition on $L^2(SF)$, we have reduced $\Delta_F$ to the family of ordinary differential operators
$\{\Delta_I\}_{I \in \calI}$.  It follows from standard ODE theory that $\mbox{spec}\,(\Delta_I) = [b_I^2, \infty)$, and that 
this spectrum is purely absolutely continuous. Our next goal is to describe the continuous spectral decomposition 
associated to each $\Delta_I$.  The fact that the threshold $b_I^2$ depends on $I$, and in particular that the set $\{b_I^2\}$
accumulates at $0$ if the cusp is irrational is the cause of the main difficulties below when summing over $I$. 
However, for the moment, we are still analyzing each operator $\Delta_I$ individually.

The spectral decomposition for $\Delta_I$ is determined by its spectral measure, which in turn is given via Stone's formula 
in terms of the resolvent. Thus consider the family of equations
\[
(\Delta_I - \la) f = 0, \qquad \la \in \cc \setminus [b_I^2, \infty). 
\]
If we conjugate with $r^{-(n-k-2)/2}$ and set $\la = t^2 + b_I^2$, then this can be recognized as a Bessel equation:
\[
\Delta_I - \la = - r^{-2} r^{-\frac{n-k-2}2} \Big((r\pl_r)^2-\Big(\frac{n-k-2}{2}+m\Big)^2+ t^2 r^2\Big) r^{\frac{n-k-2}2},
\]
so the space of homogeneous solutions is spanned by Bessel functions (see Appendix)
\[ 
r^{-\frac{n-k-2}2}J_{\frac{n-k-2}{2}+m}(rt) \qquad \mbox{and} \qquad  r^{-\frac{n-k-2}2}H^{(1)}_{\frac{n-k-2}{2}+m}(rt).
\]
The convention here is that ${\rm Im}(t) > 0$ when $\la \in \cc \setminus [b_I^2,\infty)$, which corresponds to the choice 
${\rm Im}(\sqrt{\mu}) > 0$ when $\mu\in\cc \setminus \rr^+$.  The Schwartz kernel of the resolvent thus has the explicit expression 
($H$ is the Heaviside function)
\begin{equation}
\begin{array}{l}
\qquad \qquad R_I(t;r,r')  :=  (\Delta_I - t^2 - b_I^2)^{-1}  \\[0.5ex] =  (rr')^{-\frac{n-k-2}{2}}J_{\frac{n-k-2}{2}+m}(r t ) H^{(1)}_{\frac{n-k-2}{2}+m} (r't)H(r'-r)  
\\[0.5ex] \qquad +  (rr')^{-\frac{n-k-2}{2}}H^{(1)}_{\frac{n-k-2}{2}+m}(rt) J_{\frac{n-k-2}{2}+m}(r't)H(r-r').
\end{array}
\label{resolvDeltaI}
\end{equation}
From this, Stone's formula gives the spectral measure of $\Delta_I-b_I^2$ as
\begin{equation}\label{spmeas}
\begin{split}
dE_I(t;r,r')=&\ \frac{1}{i\pi} (R_I(t;r,r')-R_I(-t;r,r'))\, t \, dt \\
=&\ \frac{2}{i\pi} (rr')^{-\frac{n-k-2}{2}}J_{\frac{n-k-2}{2}+m}(rt)J_{\frac{n-k-2}{2}+m}(r't)\, t\, dt.
\end{split}
\end{equation}

This leads to the spectral resolution of a function $f_I \in L^2(\rr^+; r^{n-k-1}dr)$: 
\[
f_I(r) = \int^{\oplus} \tilde{f}_I(r,t)\, dE_I(t), \qquad \tilde{f}_I(r,t) = \int_{r'=0}^\infty f_I(r')\, dE_I(t,r,r') dr'.
\]
In the next subsection we invoke the functional calculus to define functions of $\Delta_I$ by the formula
\begin{equation}\label{GPI}
G(\Delta_I)=\int_{0}^\infty G(t^2+b_I^2)\, dE_{I}(t).
\end{equation}
for a suitable class of functions $G(t)$. 

\section{The resolvent when $\Gamma$ is an elementary parabolic group}
We now turn to the construction and analysis of the resolvent of the Laplacian on the quotient $X_c =\Gamma_\infty \backslash \hh^{n+1}
\simeq \rr^+_x\x F$ of hyperbolic space by an elementary parabolic group $\Gamma_\infty$ fixing $\infty$ in the half-space 
model. It is convenient to work with the unitarily equivalent operator 
\[
P = x^{-n/2}\Delta_{X_c} x^{n/2} = - (x\del_x)^2 + x^2 \Delta_F + \frac{n^2}{4}
\]
acting on $L^2(\rr^+ \times F; \frac{dx}{x} dv_F)$. This decomposes into components
\begin{equation}\label{defPI}
P = \bigoplus_I P_I; \qquad P_I = - (x\del_x)^2 + x^2 \Delta_I + \frac{n^2}{4}.
\end{equation}
where $\Delta_I$ is the operator of \eqref{eq:redlapF}.
These are each symmetric on $L^2(\rr^+_x \times \rr^+_r; r^{n-k-1}\frac{dx}{x}dr)$. In the following we write
\[d\mu = r^{n-k-1}x^{-1} dx dr.\]
Using the same ODE formalism as above (i.e. Sturm-Liouville theory), along with (\ref{GPI}), we obtain 
the Schwartz kernel of the resolvent $R(s)=(P -s(n-s))^{-1}$ of $P$
\begin{equation}\label{resolvante}
\begin{split}
R(s; \, x,r\omega,z, & x',r'\omega',z') =  \\
\sum_{I} \int_{0}^\infty  & \Big(K_{s-n/2}\Big(x\sqrt{t^2+b_I^2}\Big)I_{s-n/2}\Big(x'\sqrt{t^2+b_I^2}\Big)H(x-x') \\
+ & K_{s-n/2}  \Big(x'\sqrt{t^2+b_I^2}\Big)I_{s-n/2}\Big(x\sqrt{t^2+b_I^2}\Big)H(x'-x)\Big) \\
\times \frac{i^{n-k+2m}}{\pi^2} & J_{\frac{n-k-2}{2}+m}(rt)J_{\frac{n-k-2}{2}+m}(r't)\, t\, dt \, \phi_I(z,\omega)\, \phi_I(z',\omega'),
\end{split}
\end{equation}
which, we show below, is valid when ${\rm Re}(s)>n/2$ as an operator acting on $L^2$.  

\subsection{Continuation of the resolvent to $\cc$ in weighted $L^2$ spaces}
We now show that the explicit formula (\ref{resolvante}) is the resolvent $R(s)$ of $P$ in ${\rm Re}(s)>n/2$ and that it has  
a meromorphic continuation to the entire complex $s$-plane in weighted spaces. 
\begin{prop}\label{firstcase}
The resolvent for $P$ is given in $\{{\rm Re}(s)>n/2\}$ by the expression \eqref{resolvante} as a continuous operator 
on $L^2(X_c,\frac{dx}{x}dv_F)$. If $\chi\in C_0^\infty([0,\infty)\x F)$, $N > 0$ and $\rho:=x/(x+1)$, then the operator 
$\chi R(s)\chi$ extends from the half-plane ${\rm Re} (s) > n/2$ to ${\rm Re}(s)>n/2-N$ as a 
holomorphic family of bounded operators from $\rho^NL^2(\frac{dx}{x}dv_F)$ to $\rho^{-N}L^2(\frac{dx}{x}dv_F)$. 
\label{pr:ac}
\end{prop}
\begin{proof}  By spectral theory and the fact that the essential spectrum of $P$ is $[n^2/4,\infty)$ (see \cite{LP}),
the resolvent $(P-s(n-s))^{-1}$ is meromorphic with finite rank poles (corresponding to the finite set of $L^2$ eigenvalues) in $\{{\rm Re}(s)>n/2\}$. 
In order to extend the resolvent, we first show that the expression \eqref{resolvante} is the actual resolvent in the physical half-plane 
$\{{\rm Re}(s)>n/2\}$. To prove the $L^2$ boundedness of \eqref{resolvante}, decompose $f \in L^2(X_c,\frac{dx}{x}dv_F)$ as 
$f = \sum_I f_I(x,r) \phi_I(z,\omega)$ where $f_I(x,r)\in L^2(\rr^+\times \rr^+;  d\mu)$. Then
\[R(s)f (x,r\omega,z)= \sum_{I} (R_I(s)f_I)(x,r)\phi_I(z,\omega) \]
and it suffices to show that for ${\rm Re}(s)>n/2$, 
\begin{equation}\label{toshow}
||R_I (s)f_I||_{L^2(d\mu)}\leq C||f_I ||_{L^2(d\mu)}
\end{equation}
for some $C$ independent of $I$. We thus write 
\[ R_{I}(s)f_I (x,\cdot)=\int_{0}^{\infty}F_{s,x,x'}(\sqrt{\Delta_I}\,)f_I(x',\cdot) \,\frac{dx'}{x'},\]
where 
\[F_{s,x,x'}(\tau):=K_{s-\ndemi}(x\tau)I_{s-\ndemi}(x'\tau)H(x-x') +I_{s-\ndemi}(x\tau)K_{s-\ndemi}(x'\tau)H(x'-x).\]
This function is holomorphic in $\cc$ as a function of $s$. By \eqref{besseldef} and \eqref{estbessel}, 
\begin{equation}
\label{ki}
\begin{gathered}
|K_{s-\ndemi}(x\tau)|\leq \left\{\begin{array}{ll}
C (x\tau )^{-|{\rm Re}(s)-n/2|} & \textrm{ if }x\tau \leq 1\\
C e^{-x\tau}/\sqrt{x\tau} & \textrm{ if }x\tau >1
\end{array}\right.\\
|I_{s-\ndemi}(x\tau)|\leq \left\{\begin{array}{ll}
C (x\tau)^{{\rm Re}(s)-n/2} & \textrm{ if }x\tau\leq 1\\
C e^{x\tau}/\sqrt{x\tau} & \textrm{ if }x\tau>1,
\end{array}\right.
\end{gathered}
\end{equation}
for $s \neq n/2$ (and $C$ depend on $s$). We thus estimate for ${{\rm Re}(s)\geq n/2}$
\[|F_{s,x,x'} (\tau)| \leq  
\left\{\begin{array}{ll}
C (\min(x,x')/(\max(x,x'))^\demi  & \textrm{ if both } x\tau, x'\tau  \geq 1 \\
C(\min(x,x')/\max(x,x'))^{{\rm Re}(s)-n/2} & \textrm{ if both } x\tau, x'\tau  < 1\\
C (x/x')^{{\rm Re}(s)-n/2}& \textrm{ if } x \tau <1<x'\tau \\
C(x'/x)^{{\rm Re}(s)-n/2}& \textrm{ if } x'\tau <1<x\tau. 
\end{array}\right.\]
and in particular $N_{s}(x,x'):=\sup_{\tau\in\rr^+}|F_{s,x,x'}(\tau)|$ is a kernel such that for ${\rm Re}(s)>n/2$
\[  
\sup_{x\in\rr^+}\int_0^\infty N_s(x,x')\frac{dx'}{x'}\leq C , \,\,  \sup_{x'\in\rr^+}\int_0^\infty N_s(x,x')\frac{dx}{x}\leq C
\]  
for some $C>0$ depending on $s$. By Schur's lemma, it is the kernel of a bounded operator on $L^2(dx/x)$ with norm
less or equal to $C$, which proves \eqref{toshow}.

Now we study the continuation to $s\in\cc$. We first decompose the set $\mc{I}$ of indices $mpv^*$ as
$\mc{I}=\mc{I}_{>}\cup \mc{I}_{0}$, where 
\[I \in \mc{I}_0 \iff b_I=0\]
for $b_I$ as in \eqref{eq:defCI}. Assume that $\chi(x,y,z) = \varphi(x)\psi(r)$ with $r=|y|$; this is invariant under $\Gamma_\infty$, hence 
descends to $X_c$. Again for $f \in L^2(X_c)$, $f = \sum_I f_I(x,r) \phi_I(z,\omega)$ 
with $f_I(x,r)\in L^2(\rr^+\times \rr^+;  d\mu)$, then for $N>0$ fixed, we must show
\begin{equation}\label{toshow2}
||\chi \rho^N R_I (s) \rho^N \chi f_I||_{L^2(d\mu)}\leq C||f_I ||_{L^2(d\mu)}
\end{equation}
with holomorphic dependance on $s$, where $\rho=x/(1+x)$. Note that $\rho$ can be replaced by $x$ since $x$ is bounded on ${\rm supp}(\chi)$.
We write
\[(\chi R_{I}(s)\chi f_I) (x,r)=\int_{0}^{\infty}\chi(x,r) F_{s,x,x'}(\sqrt{\Delta_I}\, )(f_I \chi) \,dx'/x',\]
and proceed as above. We estimate 
\begin{equation}
\label{fla}
|F_{s,x,x'} (\tau)| \leq  
\left\{\begin{array}{ll}
C & \textrm{ if both } x\tau, x'\tau  \geq 1 \\
C\max((xx'\tau^2)^{{\rm Re}(s)-n/2},1) & \textrm{ if both } x\tau, x'\tau  < 1\\
C \max((x \tau)^{{\rm Re}(s)-n/2},1)& \textrm{ if } x \tau <1<x'\tau \\
C\max((x'\tau )^{{\rm Re}(s)-n/2},1)& \textrm{ if } x'\tau <1<x\tau. 
\end{array}\right.
\end{equation}
where $C$ depends on $s$ only. If $x, x' \in \mbox{supp}(\varphi)$, 
${\rm Re}(s) >n/2 -N$ and $\tau \geq b_I$,  then 
\[(xx')^N |F_{s,x,x'}(\tau)| \leq C(1 + b_I^{2{\rm Re}(s)-n}), \quad \textrm{if }I\in\mc{I}_>.\]
Hence, by the spectral theorem, since $\sqrt{\Delta_I} \geq b_I$,  $(xx')^N F_{s,x,x'}(\sqrt{\Delta_I})$ is  bounded on
$L^2(\rr^+,r^{n-k-1}dr)$ with norm controlled by $ C(1 + b_I^{2{\rm Re} (s)-n})$. 

The norm blows-up when $b_{I}\to 0$ and ${\rm Re}(s)<n/2$, which makes the continuation to the non-physical half-plane delicate, since we 
need an estimate which is uniform in $I$. The key to controlling the terms for which $b_I$ are arbitrarily close (but not equal) to $0$
is that the spectral measure $dE_I(t)$ of $\Delta_I-b_I^2$ is $\calO(t^{2m+n-k-2})$ as $t\to 0$ if $I=mpv^*$. This can be seen from estimate \eqref{besselvanish}Ê below on Bessel functions and even more directly from the order of vanishing at $r=0$ of the regular solutions of the ODE
\[\Big(-(r\pl_r)^2+\Big(\frac{n-k-2}{2}+m\Big)^2-r^2\Big)u(r)=0.\]
on $[0,\infty)$. The terms for which $b_I=0$ are dealt with differently by reducing to a lower dimensional 
hyperbolic space, in some sense. Let us explain this.

Fix $N>0$ and let $\eps_N:=\min_{I\in\mc{I}_>,\, m\leq N}b_I$. Let $\calI_{N} = \{ I \in \calI_>: m \leq N\}$. By what we have just established,
if $I \in \calI_{N}$, then for $\eps>0$ small and ${\rm Re}(s)>n/2-N+\eps$
\[\begin{split}
||x^{N-\eps}\psi R_I(s)\chi x^Nf_I||_{L^2_r}\leq&  \int_{0}^\infty ||x^{N-\eps}
\psi F_{\la,x,x'}(P_I){x'}^{N-\eps}\psi f_I(x',.)||_{L^2_r}|\varphi(x')|{x'}^{\eps}\frac{dx'}{x'}\\
& \leq C\int_0^\infty {x'}^{\eps}|\varphi(x')|.||f_I(x',.)||_{L^2_r}\frac{dx'}{x'} \leq C' ||f_I||_{L^2(d\mu)}  
\end{split}\]
where $L^2_r:=L^2(\rr^+,r^{n-k-1}dr)$ and $C' = \calO(C^N/\eps_N^N))$ as $N\to \infty$. But
$\varphi(x)x^{\eps}||f_I||_{L^2(d\mu)}$ is in  $L^2(\rr^+,dx/x)$ with norm bounded by $C||f_I||_{L^2(d\mu)}$
so \eqref{toshow2} is valid uniformly for $I\in\calI_{N}$. 

The estimates on $F_{s,x,x'}$ above also imply that for any $I$ whatsoever, 
\[
||(xx')^N\indic_{[1,\infty)}(\sqrt{\Delta_I})F_{s,x,x'}(\sqrt{P_I})|| \leq C 
\]
as an operator on $L^2( r^{n-k-1}dr)$, so arguing just as above, we see that for all $I$, 
\begin{equation}\label{indic}
||x^N \chi R_I(s) \indic_{[1,\infty)}(\sqrt{\Delta_I}) \chi x^N f ||_{L^2(d\mu)} \leq C ||f_I||_{L^2(d\mu)}.
\end{equation}

The next step, therefore, is to establish the uniform bound for ${\rm Re}(s)>n/2-N+\eps$
\[
||x^{N-\eps}\psi\indic_{(0,1)}(\sqrt{\Delta_I})F_{s,x,x'}(\sqrt{\Delta_I})\chi {x'}^{N-\eps}f_I(x',\cdot )||_{L^2(r^{n-k-1}dr)} \leq C ||f_I(x',\cdot)||_{L^2(r^{n-k-1}dr)}
\]
when $m \geq N$ and $x, x' \in \mbox{supp}(\varphi)$.  Using \eqref{spmeas}, \eqref{GPI}  and the bounds \eqref{fla} 
we can estimate the Schwartz kernel for $m\geq N$
\begin{multline*}
|\psi(r)\psi(r')(xx')^{N-\eps}\indic_{[0,1]}(\sqrt{\Delta_I})F_{s,x,x'}(\sqrt{\Delta_I})(r,r')|  \\ 
\leq C\sup_{m \geq N;t\in(0,1)}|J_{\frac{n-k-2}{2}+m}(rt)J_{\frac{n-k-2}{2}+m}(r't)|(1 + t^{2{\rm Re}(s) - n}),
\end{multline*}
again uniformly in $x,x'\in\supp (\varphi)$.  But now, since $r, r' \leq C$ in $\supp \, \psi$ and $0 < t < 1$, 
we have by \eqref{estbessel}
\begin{equation}\label{besselvanish}
|J_{\frac{n-k-2}{2}+m}(rt)|\leq \frac{(rt)^{\frac{n-k-2}{2}+m}}{\Gamma(\frac{n-k}2+m)},
\end{equation}
and hence
\begin{multline*}
|\psi(r)\psi(r')(xx')^{N-\eps}\indic_{[0,1]}(\sqrt{\Delta_I})F_{s,x,x'}(\sqrt{\Delta_I})(r,r')| \\ \leq C
\sup_{m \geq N;\,  t\in(0,1)}(1 +  t^{n-k-2 + 2m - 2N}) |\psi(r)\psi(r')|,
\end{multline*}
uniformly in $x,x'\in\supp\, \varphi$. This is bounded provided $m \geq N - \frac12 (n-k-2)$. 
From this bound on the Schwartz kernel, we obtain directly that 
\begin{equation}\label{indic2}
||x^N \chi R_I(s) \indic_{[0,1]}(\sqrt{\Delta_I}) \chi x^N f ||_{L^2(d\mu)} \leq C ||f_I||_{L^2(d\mu)}.
\end{equation}
Notice that the holomorphy in $s \in \{{\rm Re}(s)>n/2-N\}$ follows immediately from the holomorphy of 
$K_{s- \ndemi}(z)$ and $I_{s-\ndemi}(z)$ when $z\in(0,\infty)$ and the fact that $|\pl_s K_s(z)|$ and $|\pl_sI_s(z)|$ satisfy the 
same type of bounds as $|K_s(z)|,|I_s(z)|$ by Cauchy's formula.\\

It remains finally to deal with the terms with indices in $\{ I=mpv^*\in \mc{I}_0; m\leq N\}$. 
If $I=mpv^*\in \mc{I}_0$, then $\Delta_I$ acting on $L^2(\rr^+,r^{n-k-1}dr)$
is unitarily equivalent to the Laplacian acting on $\rr^{n-k}$ but restricted on the subspace 
$L^2(\rr^+, r^{n-k-1}dr; H_m)$ under the decomposition $L^2(\rr^{n-k})=\oplus_{m=0}^\infty L^2(\rr^+,r^{n-k-1}dr; H_m)$.
In addition, we can rewrite
\[ 
P_I-s(n-s) = -(x\pl_x)^2+x^2\Delta_I+\frac{(n-k)^2}{4}-t(n-k-t), \quad \textrm{ with }t:=s-k/2, 
\]
and hence deduce that $x^{\frac{n-k}{2}}R_I(s)x^{-\frac{n-k}{2}}$ is unitarily equivalent, in ${\rm Re}(s)>n/2$, to 
$R_{\hh^{n-k+1}}(t):=(\Delta_{\hh^{n-k+1}}-t(n-k-t))^{-1}$ acting on $L^2(\rr^+,\frac{dx}{x^{n-k+1}}; L^2(\rr^+,r^{n-k-1}dr; H_m))$, 
with $t=s-k/2$, under the decomposition
\[
L^2(\hh^{n-k+1})=L^2(\rr^+,\frac{dx}{x^{n-k+1}}; L^2(\rr^{n-k}))\simeq \bigoplus_{m=0}^\infty 
L^2\Big(\rr^+,\frac{dx}{x^{n+1}}; L^2(\rr^+,r^{n-k-1}dr; H_m)\Big).
\]
But it is known \cite{GZ2} that the resolvent $x^N\chi R_{\hh^{n-k+1}}(t)\chi x^N$ has a meromorphic (resp. holomorphic) 
extension if $n-k+1$ is even (resp. odd), with simple poles at $t\in -\nn_0$ and finite rank residues. 
In particular, since $\chi(x,r)$ commutes with the decomposition into spherical harmonics, this implies
that $x^{N}\chi R_I(s)\chi x^N$ has a meromorphic continuation with the same property (the poles then lie 
in $k/2-\nn_0$).  The proof is now complete.
\end{proof}

We also prove a technical lemma which is useful later, the proof of which follows the same lines as
the argument above. 
\begin{lem}\label{technical}
Let  $\chi\in \calC_0^\infty([0,\infty)\x F)$; then for any $N\in\nn$, 
there exist operators $M_\ell(s):\calC_0^\infty(X_c)\to L^2(F)$ such that for any $\varphi\in \calC_0^\infty(X_c)$
\begin{equation}\label{decompos}
 (\chi R_{X_c}(s)\varphi)(x,y,z) - \chi \sum_{\ell=0}^{N}x^{s+2\ell}(M_{\ell}(s)\varphi)(y,z) \in x^{{\rm Re}(s)+2N}L^2(X_c)
 \end{equation}
and $\Gamma(s-n/2+\ell+1)M_\ell(s)$ is meromorphic in $s\in\cc$, with at most 
simple poles at $s_0\in k/2-\nn_0$, and finite rank residues. 
\end{lem}
\begin{proof}
When ${\rm Re}(s)>n/2$, $x^{-s}R_{X_c}(s)\varphi$ can be written, using \eqref{resolvante}, as 
\begin{equation}\label{expatx=0} 
(x^{-s}R_{X_c}(s)\varphi)(x,\cdot )=x^{\ndemi-s} 
\int_\eps^\infty  I_{s-\ndemi}(x\sqrt{\Delta_{F}})K_{s-\ndemi}(x'\sqrt{\Delta_{F}})(x')^{-\ndemi}\varphi(x',\cdot)\frac{dx'}{x'}
\end{equation}
for $\eps>x>0$ if $\eps>0$ is such that ${\rm supp}(\varphi)\subset \{x>\eps\}$. 
Now, for any $N\in\nn$ and $\tau\in(0,\infty)$, the modified Bessel function $I_{s-n/2}(\tau)$ satisfies 
\[
\Big| I_{s-n/2}(\tau)-(\tau/2)^{s-\ndemi}\sum_{\ell=0}^{N} \frac{2^{-2\ell}\tau^{2\ell}}{\ell!\, \Gamma(s-n/2+\ell+1)} \Big|
\leq C \min(\tau,1)^{{\rm Re}(s)+2N+2}e^{\tau}
\]
for some $C$ depending on $s$. Then, by mimicking the proof of Proposition \ref{firstcase}, we obtain directly that 
if we set 
\[ 
M_{\ell}(s)\varphi= \frac{2^{-2\ell}}{\ell!\, \Gamma(s-n/2+\ell+1)}
\int_{\eps}^\infty \psi \Delta_{F}^{\ell}K_{s-n/2}(x'\sqrt{\Delta_F})x'^{-\ndemi}\varphi(x',\cdot)\frac{dx'}{x'},
\]
then \eqref{decompos} holds, 
where $\psi\in \calC_0^\infty(F)$ and $\chi\psi=\chi$; the integral has 
meromorphic extension in $s\in\cc$ as a function in $L^2(F)$ by the same arguments as in the proof of
Proposition \ref{firstcase}. Poles can arise only when $I\in \mc{I}_0$, hence lie in  $k/2-\nn_0$, and 
their residues have finite rank.
\end{proof}

\subsection{Regularity of the Green kernel up to $\pl\bbar{X}$}\label{0calculus}
We have now established that the family of operators $R_{X_c}(s)$ has an analytic continuation to $\cc$,
albeit with a rather minimal description of the regularity of its integral kernel. Obviously, the structure of this kernel is 
standard in any compact set of $X_c$ by usual elliptic regularity, but near infinity it must be analyzed using a more
involved approach. For the analysis in any relatively compact open set of $\bbar{X}_c$, we can apply the method explained in 
\S \ref{Guillopezworski}. 

To be more precise, let $W := \{x^2 + r^2 \leq R^2\} \subset X_c$ and its partial closure $\bbar{W}_0= 
W \cup (\bbar{W} \cap \{x=0\} )\subset \bbar{X}_c$.  In what follows, $\calC_0^\infty(\bbar{W}_0)$ denotes
the set of smooth functions with compact support in $\bbar{W}_0$ (but a priori not compact in $W_0$), and similarly for 
$\bbar{W}_0\x \bbar{W}_0$.
\begin{prop}\label{smoothness}
Fix $\psi_1,\psi_2\in \calC_0^\infty(\bbar{W}_0)$ with disjoint supports; then $\psi_2R_{X_c}(s)\psi_1$ has an integral kernel which
lies in $x^s{x'}^s\calC_0^\infty(\bbar{W}_0\x \bbar{W_0})$. 
\end{prop}
\begin{proof} We actually show something slightly more refined since we give a sharp characterization of the structure of the resolvent 
kernel in $\bbar{W}_0\x \bbar{W}_0$.  We use the parametrix construction explained in \S \ref{Guillopezworski}, hence shall be brief and 
refer the reader to that discussion. Cover $W$ by finitely many open charts 
$\calU_1,\dots,\calU_\ell$ such that $W \subset \calU:=\cup_{j=1}^\ell \calU_j$, where each $\calU_j$ is identified by an isometry $\iota_j$ to the half-ball $B=\{(x,y)\in\rr^+\x\rr^n; x^2+|y|^2<1\}$ in $\hh^{n+1}$, like 
in \S \ref{Guillopezworski}. Using $\iota_j$, we can systematically identify operators on $\calU_j$ with their counterparts on $B$.
Let $\chi_j,\hat{\chi}_j\in \calC_0^\infty(\bbar{W}_0)$ such that $\chi:=\sum_{j=1}^\ell\chi_j$ is equal to $1$ on $\bbar{W}_0$, $\hat{\chi}_j$ 
is supported in $\calU_j$ and $\hat{\chi}_j=1$ on the support of $\chi_j$. By the argument of \S \ref{Guillopezworski}, for each $N\in \nn$, 
we can construct a kernel $Q_{N,j}(s),K_{N,j}(s)$ supported in $\calU_j\x \calU_j$ such that $(xx')^{-s}Q_{N,j}(s)$ and $x^{-s-2N}{x'}^{-s}K_{N,j}(s)$ 
are respectively smooth functions on $\bbar{W}_0\x\bbar{W}_0\setminus{\rm diag}$ and $\bbar{W}_0\x\bbar{W}_0$, meromorphic with finite rank simple poles contained in $-\nn_0$, and such that
\[
(\Delta_{X_c}-s(n-s))Q_{N,j}(s)=\chi_j+K_{N,j}(s)
\] 
Finally, set $Q_N(s):=\sum_{j=1}^\ell Q_{N,j}(s)$,  so that
\[
(\Delta_{X_c}-s(n-s))Q_{N}(s)=\chi+K_{N}(s), \quad K_N(s):=\sum_{j=1}^\ell K_{N,j}(s),
\]
where $K_N(s)\in x^{s+2N}{x'}^s\calC_0^\infty(\bbar{\calU}_0\x\bbar{\calU}_0)$ (here $\calU_0:=\calU\cup (\bbar{\calU}\cap\{x=0\})$ is 
the partial closure of $\calU$). 
We can go further and use Borel's lemma to construct an asymptotic limit $Q_\infty(s)$ of the $Q_N$, which satisfies
\[
(\Delta_{X_c}-s(n-s))Q_{\infty}(s)=\chi+K_{\infty}(s), \quad K_\infty(s)\in x^\infty {x'}^s 
\calC_0^\infty(\bbar{\calU}_0\x\bbar{\calU}_0)
\]

Exchanging the functions $\hat{\chi}_j$ and $\chi_j$ and applying the same method yields a right parametrix, i.e.\ operators 
$Q'_\infty(s)$ and $K'_\infty(s)$ such that 
\[
Q'_\infty(s)(\Delta_{X_c}-s(n-s))=\chi+K'_{\infty}(s), \quad K'_\infty(s)\in x^s {x'}^\infty 
\calC_0^\infty(\bbar{\calU}_0\x\bbar{\calU}_0),
\] 
where $Q'_\infty(s)$ and $K'_\infty(s)$ have the same meromorphic properties as $Q_\infty(s)$ and $K_\infty(s)$ with respect to $s$.

The error terms $K_\infty(s)$ and $K'_\infty(s)$ have Schwartz kernels with the properties that, 
for any $N>|{\rm Re}(s)-\ndemi|$, as functions of $m$, 
\begin{equation}\label{mappKinfty}
\begin{gathered}
K_\infty(s;\cdot,m) \in x^{s}\calC^\infty(\bbar{\calU}_0; x^NL^2(\calU_0)),\\
K'_\infty(s;m,\cdot) \in  x^s\calC^\infty(\bbar{\calU}_0; x^NL^2(\calU_0))
\end{gathered}
\end{equation}
where $L^2(\calU_0)$  is with respect to the hyperbolic measure on $X_c$. 
By a standard argument, $R(s)$ agrees with $Q_\infty(s)$ up to more regular term: indeed,
when ${\rm Re}(s) > n/2$, 
\[
\begin{gathered}
R_{X_c}(s)(\Delta_{X_c} - s(n-s))Q_\infty(s) = Q_\infty(s) = R_{X_c}(s)(\chi + K_\infty(s)) \\
Q'_{\infty}(s)(\Delta_{X_c} - s(n-s))R_{X_c}(s) = Q'_\infty(s) = (\chi + K'_\infty(s))R_{X_c}(s),
\end{gathered}
\]
which shows that
\[
\chi R_{X_c}(s) \chi = \chi Q_\infty(s) - Q'_\infty(s)K_\infty(s) + K'_\infty(s)R_{X_c}(s)K_\infty(s)
\]
Using Proposition \ref{firstcase} and \eqref{mappKinfty},  the last term on the right hand side  
extends meromorphically to $s\in\cc$ and has Schwartz kernel in $(xx')^s\calC^\infty(\bbar{W}_0\x\bbar{W}_0)$.
By Lemma \ref{composition}, the operator $Q'_\infty(s)K_\infty(s)$ has Schwartz kernel in 
$(xx')^s\calC^\infty(\bbar{W}_0\x\bbar{W}_0)$ and we conclude  that 
\begin{equation}\label{RX_c}
\chi R_{X_c}(s) \chi - \chi Q_\infty(s) \in x^s{x'}^s \calC_0^\infty(\bbar{\calU}_0\x \bbar{\calU}_0).
\end{equation}
for all $s\in \cc\setminus -\nn_0$. To finish the proof it suffices to note that $Q_\infty(s)$ is a sum of explicit terms, each of 
which are in $x^s{x'}^s \calC_0^\infty(\bbar{\calU}_0\x \bbar{\calU}_0\setminus {\rm diag})$ by \eqref{outsidediag}. 
\end{proof}

\section{Continuation of the resolvent on $X$}\label{seccontinu}
In this section we pass from the `local' result, i.e.\ the continuation of the resolvent on the model cusp $X_c$ to its
continuation on an arbitrarily geometrically finite quotient $X$. 

As before, let $x$ be a smooth function on $X$ which equals the upper half-space coordinate $x$ in each cusp 
neighbourhood $\calU^c_j$ (these are chosen to be disjoint) and which is a global boundary defining function on $\bbar{X}$. We also use
$\rho = x/(1+x)$, which is still a boundary defining function, but is bounded in the cusp neighbourhoods.

We now prove the following Theorem, which implies Theorem \ref{th1}. 
\begin{theo}\label{mainth}
Let $X = \Gamma \backslash \hh^{n+1}$ be geometrically finite and $R_X(s):=(\Delta_X-s(n-s))^{-1}$ the resolvent of $\Delta_X$, 
defined as a bounded operator on $L^2(X)$ for ${\rm Re}(s)>n/2$ and $s(n-s)\notin \sigma_{\rm pp}(\Delta_X)$. Fix
$\psi\in \calC_0^\infty(\bbar{X})$; then for each $N>0$, $\psi R_X(s)\psi:\rho^NL^2(X) \to \rho^{-N}L^2(X)$ extends 
as a bounded operator meromorphically to $\{{\rm Re}(s)>n/2-N\}$ with all poles of finite rank.
\end{theo}
\begin{proof}
We use the regular and cusp neighbourhoods $\calU_j^r$ and $\calU_j^c$, and a corresponding subordinate partition of unity
$\{\chi_0, \chi_j^{r/c}\}$ for this cover, with $\chi_0\in \calC_0^\infty(X)$ and $\mbox{supp}\,\chi_j^{r/c} \subset
\calU_j^{r/c}$. We also choose functions $\hat{\chi}_0\in \calC_0^\infty(X)$, $\hat{\chi}^{r/c}_j\in \calC^\infty(X)$ with
similar supports which are equal to $1$ on the supports of the corresponding `unhatted' functions.

Let $R^r_j(s)$ denote the kernel of the resolvent on $\hh^{n+1}$, restricted to a standard half-ball $B$ and transferred
back to $\calU_j^r$, and $R_j^c(s)$ the kernel of the resolvent on the model cusp $\Gamma_j \backslash \hh^{n+1}$
as constructed in the previous section, again transferred back to $\calU_j^c$. (Here $\Gamma_j$ is a representative of
the conjugacy class of the parabolic subgroup fixing the $j$-th cusp.) 

Fix a large value $s_0\gg n$ and some $\psi\in \calC^\infty(\bbar{X})$ which equals $1$ on the supports of $\nabla \hat{\chi}_0$ and
every $\nabla\hat{\chi}^{r/c}_j$ and with compact support in $\bar{X}$. As in \S 4.2, the initial parametrix is
\[ 
Q_0(s):= \hat{\chi}_0R_X(s_0)\chi_0+\sum_{j\in J^r}\hat{\chi}^r_j R^r_j(s)\chi^r_j+\sum_{j\in J^c}\hat{\chi}^c_j R^c_j(s)\chi^c_j.
\]
The point of the first term, of course, is  to capture the interior singularity of the resolvent, whereas the other terms
also capture the dependence of the parametrix in $s$ near all boundaries. This satisfies 
\[
(\Delta_X-s(n-s))Q_0(s)\psi=\psi(1+\psi L_0(s)\psi+\psi K_0(s)\psi)
\]
where
\[
\begin{gathered}
K_0(s):=\sum_{j\in J^r}K^r_{0,j}(s)+\sum_{j\in J^c}K^c_{0,j}(s)\, \textrm{ with } K_{0,j}^{r/c}:= 
\sum_{j\in J^{r/c}}[\Delta_X,\hat{\chi}^{r/c}_j] R^{r/c}_j(s)\chi^{r/c}_j, \\
L_0(s):=[\Delta_X,\hat{\chi}_0]R_X(s_0)\chi_0+(s_0(n-s_0)-s(n-s))\hat{\chi}_0R_X(s_0)\chi_0.
\end{gathered}
\]
Since $\chi_0,\hat{\chi}_0$ are compactly supported, $L_0(s)$ is compact on any weighted space $\rho^NL^2(X)$.  

The next step is to improve the part coming from the sum of the $K^r_{0,j}(s)$ as in the proof of Proposition \ref{smoothness}. 
We construct operators $Q_{N,j}^r(s)$ supported in $\calU^r_j\x \calU^r_j$ which satisfy
\[
(\Delta_X-s(n-s))Q^r_{N,j}(s)=\chi^r_j+K^r_{N,j}(s)
\] 
with $K^r_{N,j}(s)\in x^{s+2N}{x'}^s \calC^\infty(\bbar{\calU}^r_j\x \bbar{\calU}_j^r)$. Poles of $Q^r_{N,j}(s),K^r_{N,j}(s)$ lie in $-\nn_0$ 
with finite rank residues, and furthermore the error terms $K^r_{N,j}(s)$  are compact on $\rho^NL^2$. 
  
To improve the parametrix in the cusp neighbourhoods, we use that $[\Delta_X,\hat{\chi}^c_j]$ is a first order differential operator 
with coefficients supported in a neighbourhood $\calV^c_j \subset \calU^p_j$ which is relatively compact in $\bbar{X}$.
Using the coordinates $(x,y,z)$ in the model cusp,
\[
\Delta_X=-(x\pl_x)^2+nx\pl_x+x^2(\Delta_y+\Delta_z),
\]  
we can take $\hat{\chi}^p_j$ to be a smooth function of $(x^2,|y|^2)$ (in particular, independent of 
$y/|y|$ and $z$), and hence 
\begin{equation}\label{commutator}
[\Delta_X,\hat{\chi}^c_j]= x^2(a_j(x^2,y)+b_j(x^2,y)\pl_x+[\Delta_y,\hat{\chi}^c_j](x^2,y))
\end{equation} 
where $a_j$ and $b_j$ are smooth as functions of $x^2$ and $y$. Using Proposition \ref{smoothness}, we deduce directly that 
\[
K_{0,j}^c(s)=[\Delta_X,\hat{\chi}^c_j] R^c_{j}(s)\chi^c_j\in x^{s+2}{x'}^s\calC^\infty(\bbar{W}_j\x\bbar{W_j})
\]
where $W_j \subset \calU^c_j$ is some open set containing ${\rm supp}(\hat{\chi}^c_j)$ and 
$\bbar{W}_j$ is its partial closure which includes its boundary at $x=0$. 

We now claim that, with $m=(x,y,z)$, there is an expansion as $x \to 0$ of the form  
 \begin{equation}\label{expK0j}
 K_{0,j}^c(s;m,m')=\psi(m)\sum_{\ell=1}^N x^{s+2\ell}M_{j,2\ell}(s;y,z,m')+\calO(x^{{\rm Re}(s)+2N+1})
\end{equation}
for some $\psi\in \calC_0^\infty(\bbar{W}_j)$ which equals $1$ on ${\rm supp}(\hat{\chi}^c_j)$, and with coefficient
functions $M_{j,2\ell}(s)\in {x'}^s \calC_0^\infty((\bbar{W}_j\cap\{x=0\})\x \bbar{W}_j)$ such that 
$\Gamma(s-n/2+\ell)M_{j,2\ell}(s)$
is meromorphic in $s\in \cc$ with at most simple poles at contained in $\ndemi-\demi\nn$, 
and the residues are kernels of some finite rank operators. 
Indeed, by Proposition \ref{smoothness},  we know that $x^{-s-2}K^c_{0,j}(s;m,m')$ has an expansion at $x \to 0$ in powers $x^{\ell}$ with coefficients $M_{j,\ell}(s)\in {x'}^s \calC_0^\infty((\bbar{W}_j\cap\{x=0\})\x \bbar{W}_j)$, which are meromorphic in $s$. The $M_{j,\ell}(s)$ are Schwartz kernels of some operators, and uniquely determined
by the choice of $x$, it remains to prove that $M_{j,2\ell+1}(s)=0$ 
and that $\Gamma(s-n/2+\ell)M_{j,2\ell}(s)$ has finite rank poles (contained in $\ndemi-\demi\nn$).
Take any $\varphi\in \calC_0^\infty(\calU^c_j)$ and $\varphi'\in \calC_0^\infty(\bbar{W}_j)$, then 
Lemma \ref{technical} shows that $x^{-s}\varphi' (R^c_j(s)\varphi)
\sim \sum_{\ell}x^{2\ell}N_{j,2\ell}(s)\varphi$ at $x=0$, for some operators $N_{j,2\ell}(s)$  
mapping $\calC_0^\infty(\calU_j)$ to $L^2(\bbar{\calU}_j\cap\{x=0\})$ 
and with $\Gamma(s-n/2+\ell)N_{j,2\ell}(s)$ having only simple poles of finite rank. 
By the expression in \eqref{commutator}, the same is true for $x^{-s-2}\varphi' K^c_{0,j}(s)\varphi$ (but mapping to $H^{-1}$ instead of $L^2$) and this proves the claim.

As before, we solve away the expansion of $K^c_j(s;m,m')$ at $x=0$: for any $k\in\nn_0$ and $F\in \calC^\infty 
(\bbar{W}_j)$ which depends smoothly on $x^2$, we have
\begin{equation} \label{indicialeq2}
(\Delta_{X}-s(n-s))\frac{x^{s+2k}F(x^2,y,z)}{2k(n-2s-2k)}=x^{s+2k}F(x^2,y,z)+\frac{x^{s+2k+2}H_k(x^2,y,z)}{2k(n-2s-2k)}
\end{equation} 
with $H_{k}\in \calC^\infty(\bbar{W}_j)$ smooth in $x^2$ and independent of $s$. Combining \eqref{indicialeq2} with 
\eqref{expK0j} and the fact that each $\Gamma(s-n/2-\ell)M_{j,\ell}(s)$ has at most first order poles with finite rank residues, 
we can construct for all $N\in\nn$ an operator $Q^c_{N,j}(s)$ which is holomorphic in $\{{\rm Re}(s)>n/2-N\}$ with
$Q^c_{N,j}(s)-\hat{\chi}^c_j R^c_j(s)\chi^c_j\in x^{s+2N}(x')^s\calC_0^\infty(\bbar{W}_j\x\bbar{W}_j)$ and 
\[
(\Delta_X-s(n-s))Q^c_{N,j}(s)=\chi^c_j+K^c_{N,j}(s)
\] 
for some $K^c_{N,j}(s)\in x^{s+2N}{x'}^s \calC_0^\infty(\bbar{W}_j\x\bbar{W}_j)$  holomorphic in $\{{\rm Re}(s)>n/2-N\}$. 

We finally obtain a good parametrix
\[
Q_N(s):=\sum_{j\in J^r} Q^r_{N,j}(s)+\sum_{j\in J^c}Q^c_{N,j}(s)+\hat{\chi}_0R(s_0)\chi_0
\]
since 
\[
\begin{gathered}
(\Delta_X-s(n-s))Q_N(s)\psi = \psi(1+\psi(L_0(s)+ K_N(s))\psi)=: \psi(1+\til{K}_N(s)) \\
\textrm{ with }K_N(s)= \sum_{j\in J^c} K^c_{N,j}(s)+\sum_{j\in J^r}K^r_{N,j}(s) \in x^{s+2N}{x'}^{s}\calC_0^\infty(\bbar{X}\x\bbar{X}).
\end{gathered} 
\]
Notice that the support of $K_N(s)$ is compactly supported and does not intersect the cusps in either set of variables.
Hence $\til{K}_N(s)$ is compact on $\rho^NL^2(X) \subset  L^2(X)$ and meromorphic in $\{{\rm Re}(s)>n/2-N\}$ with poles 
of finite multiplicy. Using standard arguments, we can modify $Q_N(s)$ by a finite rank operator if $1+\til{K}_N(s_0)$ is not 
invertible, to make the new remainder invertible, and this can be done without changing the regularity 
properties of $Q_N(s),\til{K}_N(s)$. 

We  then invoke the analytic Fredholm theorem to show that $(1+\til{K}_N(s))^{-1}$ has a meromorphic extension 
to $\{{\rm Re}(s)>n/2-N\}$ with poles of finite multiplicity, as an operator bounded on $\rho^NL^2(X)$. Thus
\[ 
\psi R_X(s)\psi= \psi Q_N(s)\psi(1+\til{K}_N(s))^{-1}
\]
gives the meromorphic extension of the resolvent in $s\in \{{\rm Re}(s)>n/2-N; s\notin n/2-\nn\}$. 
as an operator from $\rho^NL^2(X)$ to $\rho^{-N}L^2(X)$.

As in the proof of Prop. \ref{smoothness}, we can obtain the extension to all of $\cc$ directly, rather than only to any half-plane,
using Borel summation to solve away the entire expansion as $x \to 0$, and a standard pseudodifferential 
parametrix construction to correct the compactly supported error part $L_0(s)$.  This yields $Q_\infty(s)$, with 
the all same properties as $Q_N(s)$, which satisfies
\[
(\Delta_X-s(n-s))Q_\infty(s)\psi = \psi(1+  K_\infty(s))
\]
where $K_\infty(s)\in \rho^{\infty}{\rho'}^{s} \calC_0^\infty(\bbar{X}\x\bbar{X})$ is a residual term with support contained 
in ${\rm supp}(\psi)\x {\rm supp}(\psi)$ and $(1+\psi K_\infty(s_0)\psi)$ invertible.  By the analytic Fredholm
theorem again, 
\[
R_X(s)\psi=  Q_\infty(s)\psi(1+ K_\infty(s))^{-1}
\]
and this is meromorphic in $s\in\cc$ with poles of finite multiplicity. Then $(1+K_\infty(s))^{-1}=1+S_\infty(s)$  where
$S_\infty(s)=-K_\infty(s)+K_\infty(s)(1+K_\infty(s))^{-1}K_\infty(s)$ has a kernel in in $\rho^{\infty}{\rho'}^{s}
\calC^\infty(\bbar{X}\x\bbar{X})$ and has support contained in ${\rm supp}(\psi)\x {\rm supp}(\psi)$. This gives that
\begin{equation}\label{formulaext}
R_X(s)\psi =Q_\infty(s)\psi +Q_\infty(s)\psi S_\infty(s)
\end{equation}
and the operator $R_X(s)\psi$  has the same mapping properties as $Q_\infty(s)\psi$. 
\end{proof}

\begin{remark}\label{moregeneral} 
Let $(X,g)$ be a manifold which admits a decomposition $X=K\cup_{i\in J^c} E^c_j
\cup_{j\in J^r} E_j^r$  where $(K,g)$ is a smooth compact manifold with boundary, $(E^c_j,g)$ are isometric to standard 
cusp neighbourhoods (and thus have constant curvature) and $(E^r_j,g)$ are isometric to 
\[ 
\{(x,y)\in\rr^+\x \rr^n; x^2+|y|^2\leq 1\} \textrm{ with metric }g=(dx^2+h(x^2))/x^2
\]
where $u\in[0,1]\to h(u)$ is a one parameter smooth family of tensors on $\{|y|\leq 1\}$. Combining this with the
parametrix construction from \cite{MM} -and \cite{GuiDMJ} for the issue about the points $n/2-\nn$-, 
the same proof as above yields the meromorphic continuation of the 
resolvent $R(s)$ of $\Delta_g$ to $s\in\cc$, with poles of finite rank.  
\end{remark}

\section{Finer description of the resolvent}
In Theorem \ref{mainth}, we have shown that the resolvent $R_X(s)$ continues meromorphically in $s$ as an operator acting on weighted $L^2$ spaces. 
As part of this, we obtained detailed information about the Schwartz kernel $R_X(s;m;m')$ on any compact region 
$K\subset \bbar{X} \x\bbar{X}$.  We now show how to obtain alternate descriptions of $R_X(s)$ valid in certain regions
of $\cc$ and which are more precise in certain asymptotic regimes: first, we obtain a representation of $R_X(s)$ as a sum 
of translates by group elements of the free space resolvent, which converges when ${\rm Re}(s)>(n-1)/2$ and provides
good asymptotics in the cusp region; after that, we examine its Fourier analytic description more closely to 
obtain better information about asymptotics on $\partial \bbar{X}$ when $s$ lies in the closed half-plane $\{{\rm Re}(s)\geq n/2\}$. 

\subsection{The resolvent $R_X(s)$ as a sum over $\Gamma$}
Let $X_c=\Gamma_\infty \backslash  \hh^{n+1}$ with $\Gamma_\infty$ an elementary parabolic group of rank $k$, fixing $\infty$, 
with generators $(\gamma_1,\dots,\gamma_k)$. As in \S 2, we assume (by passing to a finite index subgroup) that each 
$\gamma_j$ acts on $(x,y,z) \in \hh^{n+1}$ by $\gamma_j(x,y,z) = (x,A_jy , z+v_j)$ where $v_j\in \rr^k$ and $A_j\in SO(n-k)$.
For any $\gamma \in \Gamma_\infty$, we also write 
\begin{equation}\label{gammaxyz}
\gamma (x,y,z) = (x, A_\gamma y, z + v_\gamma)
\end{equation} 
where $v_\gamma$
is in the lattice generated by the $v_j$ and $A_\gamma\in {\rm SO}(n-k)$.  A fundamental domain for this action is $\mc{F} = \rr^+ \times \rr^{n-k} 
\times \calF_T$, where $\calF_T$ is a (compact) fundamental Dirichlet domain for the induced lattice on $\rr^k$. 
We sometimes abuse notation by identifying a point $w \in \mc{F}$ with its image in $X_c$. Both $x$ and $r = |y|$ descend to $X_c$. 
\begin{prop}\label{poincareseries1}
If $w, w' \in \mc{F}$, then the resolvent kernel $R_{X_c}(s;w,w')$ can be written as
\begin{equation}\label{average}
R_{X_c}(s;w,w') = \sum_{\gamma \in \Gamma_\infty} R_{\hh^{n+1}} (s; w, \gamma w');
\end{equation}
this converges in $\calC^\infty$ (apart from the diagonal singularity) locally on compact sets of $\mc{F} \times \mc{F}$ 
for ${\rm Re}(s) > k/2$ and agrees with the continuation of $R_{X_c}(s;w,w')$ there. 
If $\chi \in \calC^{\infty}_0(\bbar{X}_c)$ and $\psi\in \calC_b^\infty(\bbar{X}_c)$ have disjoint 
supports, then for ${\rm Re}(s) > k/2$ and any $j \geq 0$ and multi-indices $\alpha,\beta$, there exists $C>0$ such that
\[
|\partial_x^j \partial_y^\alpha \partial_z^\beta [(xx')^{-s}\psi(w)R_{X_c}(s;w,w')\chi(w')]| \leq C 
(1 + x^2 + |y|^2)^{(k-|\alpha|-|\beta|-j)/2 - \mathrm{Re}(s)}.
\] 
\end{prop}
\begin{proof}  
If $w = (x,y,z), w' = (x',y',z') \in \mc{F}$ and $d(w,w')$ is hyperbolic distance, then 
\[
\cosh d(w,w') = \frac{x^2 + {x'}^2 + |y-y'|^2 + |z-z'|^2}{2xx'} := \frac{1}{\theta(w,w')}.
\]
Furthermore, by \eqref{formulemodel}, the resolvent can be written as
\begin{equation}
R_{\hh^{n+1}}(s;w,w') = \theta(w,w')^s F_s(\theta(w,w')),
\end{equation}
where $F_s \in \calC^\infty([0,1))$ depends holomorphically on $s$ in $\mbox{Re}(s) > 0$.
We use the notation \eqref{gammaxyz}.
Note that there exists a subset $\Gamma_\infty' \subset \Gamma_\infty$ with $\Gamma_\infty \setminus \Gamma_\infty'$ finite and an 
$\eps>0$ such that $|z - z' - z_\gamma|^2 \geq \eps|z_\gamma|^2$ for all $\gamma \in \Gamma_\infty'$ and $w,w'\in\mc{F}$. This gives that 
for $\gamma\in \Gamma'_\infty$
\[
\begin{split}
\frac{1}{\theta(w,\gamma w')}-1=&\frac{(x-x')^2+ |y-A_\gamma y'|^2+|z-z'-z_\gamma|^2}{2xx'}\\
\geq  &\frac{(x-x')^2+|y-A_\gamma y'|^2+\eps^2\cjg z_\gamma\cjd^2}{2xx'}.
\end{split}
\]
This is bounded below uniformly by a positive constant depending on $L$ if $x'\leq L$ and so 
is $\theta(w,\gamma w')-1$.
Consequently, for $x'\leq L$, we get 
\begin{equation}\label{seriesright}
\begin{split}
\Big|(xx')^{-s} & \sum_{\gamma \in \Gamma_\infty} R_{\hh^{n+1}} (s; w, \gamma w')\Big| \leq 
 C \sum_{\gamma \in \Gamma'_\infty} \left( x^2 + {x'}^2 + |y - A_\gamma y'|^2 + |z - z' - z_\gamma|^2\right)^{-{\rm Re}(s)}
\end{split}
\end{equation}
for some $C$ depending on $s$ and $L$ only. If we assume that $(w,w')\in\mc{F}\x\mc{F}$ are such that
${x'}^2+|y'|^2\leq L^2$ and $x^2+|y|^2\leq 4L^2$, then the series on the right in \eqref{seriesright} converges uniformly in $w,w'$
as long as ${\rm Re}(s)>k/2$, and for $w$ and $w'$ subject to these constraints, is bounded by a constant. 
If $(w,w')\in\mc{F}\x\mc{F}$ are such that
${x'}^2+|y'|^2\leq L^2$ and $x^2+|y|^2> 4L^2$, then we obtain the estimate for ${\rm Re}(s)>k/2$
\begin{equation}\label{boundabove}
\begin{split} 
\Big|(x x')^{-s}\sum_{\gamma \in \Gamma_\infty} R_{\hh^{n+1}} (s; w, \gamma w')\Big| \leq & C
(1+x^2+|y|^2)^{-{\rm Re}(s)}\sum_{\gamma\in\Gamma'_\infty} \Big(1+\frac{\cjg z_\gamma\cjd^2}{x^2+|y|^2}\Big)^{-{\rm Re}(s)}\\
 \leq & C (1+x^2+|y|^2)^{k/2-{\rm Re}(s)}
\end{split}\end{equation}
for some $C$ depending only on $s,L,L'$. Here we have used the fact that $z_\gamma$ run over a 
lattice in $\rr^k$ to obtain the bound in the second line. The same type of bounds 
are easily obtained for derivatives of any order
with respect to $w,w'$ in compact sets of $\mc{F}\x \{w'\in\mc{F}; {x'}^2+|y'|^2\leq L^2\}$, 
we omit the details. 

For any $\chi\in L^\infty(\mc{F})$ supported in 
${x'}^2+|y'|^2\leq L^2$, the operator with kernel $(xx')^{{\rm Re}(s)}(1+x^2+|y|^2)^{-{\rm Re}(s)}\chi(w')$
is bounded as an operator on $L^2(\mc{F})$ for ${\rm Re}(s)>n/2$. Therefore, since the integral kernel
$K_s(w,w'):=\sum_{\gamma\in\Gamma'_\infty}R_{\hh^{n+1}}(s;w,\gamma w')\chi(w')$ is bounded by \eqref{boundabove}, the operator with kernel $K_s(w,w')$ is bounded on $L^2(\mc{F})$. 
Moreover $K_{s}(w,w')$ clearly solves $(\Delta_{\hh^{n+1}}-s(n-s))
K_{s}(w,w')=0$ for $w,w'$ in compact sets of $\mc{F}\x \{w'\in\mc{F}; (x')^2+|y'|^2\leq L^2\}$.  

There remains to analyze $K'_s(w,w'):=\sum_{\gamma\in \Gamma_\infty\setminus\Gamma'_\infty}R_{\hh^{n+1}}(s;w,\gamma w')\chi(w')$ 
which contains the diagonal singularity (at least in the region where $(x')^2+|y'|^2\leq L^2$).
By the fact that $R_{\hh^{n+1}}(s)$ is bounded on $L^2(\hh^{n+1})$ for ${\rm Re}(s)>n/2$ and that $\Gamma_\infty
\setminus\Gamma'_\infty$ is finite, it is clear that $K'_s(m,m')$ is the kernel of a bounded operator on $L^2(\mc{F})$, moreover
it solves, in the distribution sense, $(\Delta_{\hh^{n+1}}-s(n-s))K'_s(w,w')\chi(w')=\delta(w-w')\chi(w')$.  

Combining the discussions for $K_s$ and $K_s'$, we have thus proved that $R_{X_c}(s)\chi$ has Schwartz kernel given by 
$K_s+K'_s$ in ${\rm Re}(s)>n/2$, and moreover that $K_s(w,w')+K'_s(w,w')$ is well defined as a locally uniformly converging series
on compact sets in $(m,m')$ away from the diagonal, as long as ${\rm Re}(s)>k/2$.
The last statement about the estimate of the kernel of $\psi R_{X_c}(s)\chi$ is a straightforward consequence of what we have just discussed.
\end{proof}

Now use the inverted coordinate system $(u,v,z)$ and the polar variable $R$, where
\begin{equation}\label{coorduv}
 u:=\frac{x}{x^2+|y|^2}, \quad v:= \frac{-y}{x^2+|y|^2}, \quad R:=\sqrt{u^2+|v|^2}=\frac{1}{\sqrt{x^2+|y|^2}}.
\end{equation} 
These functions too descend to $X_c$, and the cusp itself is at $R=0$, while $\partial \bbar{X}_c$ is 
$\{u=0,v\neq 0\}$. A simple calculation gives that
\[
\frac{1}{\theta(w,w')} = \frac{ u^2 + {u'}^2 + |y-y'|^2 + (RR')^2 |z-z'|^2}{2uu'}.
\]
Therefore, repeating the same arguments of the proof of Proposition \ref{poincareseries1}, we easily obtain that for any $\eps>0$, there exists $C$ with
\begin{equation}\label{estimatewithuv}
| (uu')^{-s}R_{X_c}(s;w,w')| \leq C (RR')^{-k},
\end{equation}
in the region  $\{(u-u')^2+(|v|-|v'|)^2>\eps\}$, 
with the analogous estimate holding if we apply any number of derivatives $\partial_u$,
$\partial_v$ and $\partial_z$. 

We can now use this to estimate the kernel of the resolvent on the entire geometrically finite quotient $X$.

\begin{cor}\label{corRe(s)>k/2}
Let $X = \Gamma \backslash \hh^{n+1}$ be a geometrically finite quotient, and let $\bar{k}$ be the
maximum rank of all the nonmaximal rank cusps. Let $R_X(s)$ be the meromorphically continued resolvent of Theorem \ref{mainth}. If $\chi \in \calC_0^\infty(\bbar{X})$ and $\psi\in \calC^\infty(\bbar{X})$ 
have disjoint support,
then in a cusp neighbourhood $\calU_j^c$ of a cusp of rank $k$, and using the inverted coordinates of
\eqref{coorduv}, we have for ${\rm Re}(s)> \bar{k}/2$
\begin{equation}\label{growthincusp}
\begin{gathered} 
\Big|(uu')^{-s} \psi(w)R_X(s;w,w')\chi(w')\Big|\leq C (u^2+|v|^2)^{-k/2}, \\
\Big|(uu')^{-s}\psi(w) (u\pl_u - s)R_X(s;w,w')\chi(w')\Big|\leq C (u^2+|v|^2)^{-k/2}.
\end{gathered}
\end{equation}
\end{cor}
\begin{proof}  
From \eqref{formulaext} we can write $R_X(s)\chi=Q_\infty(s)\chi +Q_\infty(s)\chi S_\infty(s)$ for a suitably
chosen parametrix $Q_\infty(s)$ and with $S_\infty(s)\in \rho^\infty {\rho'}^s \calC_0^\infty(\bbar{X} \times \bbar{X})$. 
By the construction in the proof of Theorem \ref{mainth}, if $\hat{\chi}^c_jR^c_j(s)\chi_j^c$ are the model resolvents 
in $\calU^c_j$, then $L(s) :=Q_\infty(s)-\sum_{j\in J^c}\hat{\chi}^c_jR^c_j(s)\chi^c_j$ has compact support 
in $\bbar{X}\x \bbar{X}$.  Let $\eps>0$ be small so that $\chi(w')$ has support in 
${u'}^2+|v'|^2>2\eps$ in each cusp neighbourhood, we can then combine the result in Proposition \ref{smoothness} 
for the region $u^2+|v|^2>\eps$ far from the cusps
and the estimate \eqref{estimatewithuv} for the region $u^2+|v|^2\leq \eps$ near the cusp
to deduce that each $\psi R^c_j\chi$ satisfies \eqref{growthincusp}, and $\psi L(s)\in \rho^{s+2} {\rho'}^s\calC_0^\infty(\bbar{X} \times
\bbar{X})$ satisfies this same estimate in $\calU^c_j$ as well. Finally, using the residual structure of $S_\infty(s)$,
we obtain the same estimate also for the composition $\psi Q_\infty(s)\chi S_\infty(s)$.
\end{proof}

\subsection{Poincar\'e series}
We now review a standard argument which shows that the meromorphic continuation of the resolvent $R_X(s)$
implies a corresponding extension for the Poincar\'e series of $\Gamma$: 
\[ 
P_s(m,m'):=\sum_{\gamma\in \Gamma\setminus{\rm Id}}e^{-sd(m,\gamma m') }, \quad m,m'\in\hh^{n+1}.
\]  
Here, notice that we have  removed the term in the series corresponding to $\gamma = \mbox{Id}$, 
this obviously does not change any result about meromorphic continuation of $P_s$ and is 
only done for notational simplicity below.
Recall that this sum converges to a holomorphic function in ${\rm Re}(s)>\delta$, where $\delta = \delta(\Gamma) \in (0,n)$
equals the Hausdorff dimension of the limit set of $\Gamma$ (\cite{PatActa,SU}). 
\begin{theo}
The series $P_s(m,m')$ admits a meromorphic continuation to the entire complex plane.
\end{theo}
\begin{proof} To simplify exposition, we suppose that $m,m'\in \rm{int}(\mc{F})$ where $\mc{F}$ is a fundamental domain of $\Gamma$. 
Define, for ${\rm Re}(s)>n$, 
\[\widetilde{R}_s(m,m'):= R_X(s;m,m')-R_{\hh^{n+1}}(s;m,m')=\sum_{\gamma\in\Gamma\setminus{\rm Id}} R_{\hh^{n+1}}(s;m,\gamma m').\] 
By Theorem \ref{mainth}, this extends meromorphically to $\cc$, and (by elliptic regularity) is smooth in 
$\rm{int}(\mc{F})\x\rm{int}(\mc{F})$. Now, for any $N\in\nn$, we can write by \eqref{formulemodel}
\[
R_{\hh^{n+1}}(s;m,m')=\sum_{j=0}^Nc_{s,j}Q^{s+j}(m,m')+Q^{s+N+1}(m,m')L_s(Q(m,m'))
\]
where $Q(m,m')=e^{-d(m,m')}$, and the scalar functions $c_{s,j}$ and $L_s\in \calC^\infty([0,1))$ are meromorphic in $\cc$ with 
$c_{s,0}\not\equiv 0$.  Now sum over translates by $\gamma \in \Gamma \setminus {\rm Id}$:
\begin{multline*}
P_s(m,m')  =\\ c_{s,0}^{-1} \Big( \widetilde{R}_s(m,m') -\sum_{j=1}^N c_{s,j} P_{s+j}(m,m')
- \sum_{\gamma\in\Gamma\setminus{\rm Id}} Q^{s+N+1}(m,\gamma m')L_s(Q(m,\gamma m')) \Big),
\end{multline*}
initially at least for ${\rm Re}(s)>n$. The infinite series on the right converges to a meromorphic function in ${\rm Re}(s)>n-N-1$.
Assuming that $P_s(m,m')$ is meromorphic in ${\rm Re}(s)>n-M$ for some $M \geq 0$, then all terms on the right are
meromorphic in ${\rm Re}(s)>n-M-1$, provided $N > M+1$. By induction, this provides the continuation of $P_s(m,m')$ to all of $\cc$. 
\end{proof}

We recall the result proved by Patterson in \cite[Th.1 and 2]{PatArx}. 
\begin{theo}[Patterson]
Let $X=\Gamma\backslash \hh^{n+1}$ be a geometrically finite hyperbolic manifold and $\delta>0$ be the exponent of convergence of 
Poincar\'e series. Assume that the resolvent $R_X(s)$ extends meromorphically to a neighbourhood of $\{{\rm Re}(s)\geq \delta\}$. 
Then $\Gamma(s-n/2+1)R_X(s)$ has a simple pole at $s=\delta$ and no other poles on the line $\{{\rm Re}(s)=\delta\}$, there exists a smooth function $F$ on $X$ such that the residue of 
$\Gamma(s-n/2+1)R_X(s)$ at $\delta$ is the rank $1$ operator $F\otimes F$ and 
\[
\sharp \{\gamma\in \Gamma; d(m,\gamma m')\leq R\}\sim c\,e^{\delta R}F(m)F(m')
\]
as $R \to \infty$, for some $c>0$ depending only on $\Gamma$.
\label{Patterson}
\end{theo}
\noindent Theorems \ref{mainth} and \ref{Patterson} imply Corollary \ref{cor1}. In \cite{PatArx} the function $F$ is given (up to a positive 
multiplicative constant) by $F(m)=\int_{S^n} P(m,\zeta)^\delta d\mu_{\delta}(\zeta)$
where $\mu_\delta$ is a probability measure supported on the limit set $\Lambda(\Gamma)$ of $\Gamma$ (the so-called \emph{Patterson-Sullivan measure}) and $P(m,\zeta)$ is the usual Poisson kernel on the ball. In the proof of this result by Patterson, there is a fundamental argument from ergodic theory to show that there is only one pole on the vertical line $\{{\rm Re}(s)=\delta\}$. 
Notice that only when $s\notin n/2-\nn$ the point $s=\delta$ is a resonance and  the function $F$ is then the resonant state associated to $\delta$, see e.g. \cite{GuiNau} for a discussion in the convex co-compact setting.

\subsection{Mapping properties of the resolvent} 
We now come back to the description of the resolvent on a cusp neighbourhood using Fourier decomposition,
as in \S 3, to obtain finer mapping properties of $R_X(s)$ all the way up to the critical line, i.e. in the
closed half-plane $\{{\rm Re}(s)\geq n/2\}$. This is necessary in order to analyze the scattering operator 
and prove that it satisfies a functional equation. 

Let $X_c = \Gamma_\infty \backslash \hh^{n+1} = \rr^+ \times F$ be a cusp of rank $k$. We use standard coordinates $(x,y,z)$
as before, and the Sobolev spaces $H^\ell(F)$, $H^\infty(F)$, and $H^\infty(X_c) = \calC^\infty_b([0,\infty); H^\infty(F))$, as
defined at the end of \S 2. We also use the variable $\rho = x/(1+x)$, and set
\[\dot{H}^\infty(X_c):=\{u\in H^\infty(X_c): u=\calO( (x/(1+x^2))^\infty )\}.\]
\begin{lem}\label{borneA>0} 
Let ${\rm Re}(s-\ndemi)\geq 0$ and $f\in \dot{H}^\infty(X_c)$; then $R_{X_c}(s)f \in \rho^{s-\ndemi}H^\infty(X_c)$.
\end{lem}
\begin{proof} Using the Fourier decomposition of Section \ref{fourierdec}, any $f \in H^\infty(X_c)$ decomposes
as $f = \sum_I f_I(x,r) \phi_I(z)$, $\phi_I(z) := \exp (2\pi i \langle z, v^* + A_{mp}\rangle)$, 
where each $I = (m,p,v^*)$, and $f_I(x,r) \in \Pi_{mp}(L^2(SF))$ for every $x$ and $r$. 

If $f\in L^2(X_c)$, then $R(s)f=\sum_I (R_I(s) f_I)(x,r) \phi_I(z,\omega)=\sum_I u_I(x,r) \phi_I(z,\omega)$, where
\begin{equation}\label{formulaFI}
\begin{split}
u_I(x,r)=& \int_x^\infty I_\la(x\sqrt{\Delta_I}) K_\la(x'\sqrt{\Delta_I})f_I(x',\cdot)\frac{dx'}{x'}+  \int_0^x K_\la(x\sqrt{\Delta_I}) I_\la(x'\sqrt{\Delta_I})f_I(x',\cdot)\frac{dx'}{x'},
\end{split}
\end{equation}
and we have set $\la :=s-n/2$. 
This can be rewritten using the Fourier transform $\mc{F}$ in $y=r\omega$, with dual variable $\xi$, as
\begin{equation}\label{fourierversion}
\begin{split}
 u_I(x,y)=&\Pi_{mp}\int_x^\infty \int_{\rr^{n-k}}e^{iy.\xi}I_\la(x\sqrt{|\xi|^2+b_I^2}) K_\la(x'\sqrt{|\xi|^2+b_I^2})\mc{F}(f_I)(x',\xi)d\xi
\frac{dx'}{x'}\\
&+  \Pi_{mp}\int_0^x  \int_{\rr^{n-k}}e^{iy.\xi} K_\la(x\sqrt{|\xi|^2+b_I^2}) I_\la(x'\sqrt{|\xi|^2+b_I^2})\mc{F}(f_I)(x',\xi)d\xi \frac{dx'}{x'}\\
=& u_I^< + u_I^>.
\end{split}\end{equation}
Integrating by parts yields
\begin{equation}\label{intbyparts}
\begin{gathered} 
\pl_y^{\alpha} (u_I^<)(x,y)=\\  \int_x^\infty \int_{\rr^{n-k}}e^{iy.\xi}\Big(I_\la(x\sqrt{|\xi|^2+b_I^2}) K_\la(x'\sqrt{|\xi|^2+b_I^2})
 \mc{F}((-\pl_y)^{\alpha}f_I)(x',\xi) \Big)d\xi \frac{dx'}{x'},
\end{gathered}
\end{equation}
with a corresponding identity for $u_I^>$. 

To obtain $L^2$ bounds in $y$ of $\pl_y^{\alpha} \pl_x^\beta(\rho^{-\la}u_I^<)(x,y)$, we must bound 
\[ 
\pl_x^{\beta_1}\Big(\rho^{-\la}I_\la(x\sqrt{|\xi|^2+b_I^2}))\Big) \pl_{x'}^{\beta_2} K_\la(x'\sqrt{|\xi|^2+b_I^2}),  \quad \beta_1+\beta_2\leq \beta
\]
when $x\leq x'$. We use the estimates on Bessel functions:
\begin{equation}\label{estimbesselIK}
| \pl_t^\alpha K_\la(t)|\leq \left\{
\begin{array}{ll}
C_\alpha  e^{-t} & \textrm{ if } t>1\\
C_\alpha t^{-|{\rm Re}(\la)|-\alpha} & \textrm{ if }t\leq 1
\end{array}\right., 
\quad | \pl_t^\alpha (t^{-\la}I_\la(t))|\leq  C_\alpha e^{t},
\end{equation}
valid for $t \in\rr^+$ and ${\rm Re}(\la)= A\geq 0$:  
thus when $x\leq x'$, we have
\begin{equation}\label{caseA>0}
\begin{gathered}
\Big|\pl_x^{\beta_1}\Big(\rho^{-\la}I_\la(x\sqrt{|\xi|^2+b_I^2})\Big) \pl_{x'}^{\beta_2} K_\la(x'\sqrt{|\xi|^2+b_I^2})\Big|
\leq C\cjg|\xi|+b_I\cjd^{A+\beta} \rho'^{-A-\beta},
\end{gathered}
\end{equation}
where $C$ depends on $A$ and $\beta$. Using \eqref{intbyparts} and \eqref{caseA>0}. 
Using Cauchy-Schwarz in the $x'$ integral, we immediately obtain the bound
\begin{equation}\label{L2estimate}
\begin{gathered}
||\pl_y^{\alpha} \pl_x^\beta(\rho^{-\la}u_I^<)(x,y)||_{L^2(dy)}\leq C \max_{\beta'\leq \beta}
\Big|\Big|  \frac{\cjg x\cjd^{A+\eps}}{\rho^{A+\beta+\eps}} (\Delta_F + 1)^{\frac{A+\beta}{2}}
\pl_x^{\beta'}\pl_y^\alpha f_I\Big|\Big|_{L^\infty_xL^2_F}  
\end{gathered}
\end{equation}
 when 
$\eps>0$, and ${\rm Re}(\la)=A\geq 0$; here $C$ now depends on $K,\beta$ and $A$, and we have set 
$L^\infty_xL^2_F:=L^\infty(\rr^+;L^2(F,dv_F))$. 
 
The estimates are similar for $u_I^>$, so we omit the details.
\end{proof}
The corresponding fact for the resolvent of $\Delta_X$ is a direct consequence:
\begin{cor}\label{propertyR(s)}
Let $X=\Gamma\backslash \hh^{n+1}$ be a geometrically finite quotient, and let ${\rm Re}(s-\ndemi)\geq 0$ and $f\in \dot{H}^\infty(X)$;
then $R_{X}(s)f \in \rho^{s-\ndemi}H^\infty(X)$.
\end{cor}
\begin{proof} The proof follows from Lemma \ref{caseA>0} and the parametrix construction in \eqref{formulaext}, just as 
in the proof of Corollary \ref{corRe(s)>k/2}. The crucial fact is that $R_X(s)$ equals $R^c_{j}(s)$ in the cusp neighbourhood
$\calU^j_c$, up to very residual terms.
\end{proof}

\section{Scattering theory}
Using the estimates and various properties of the resolvent we have obtained above, we now construct the 
Eisenstein (or Poisson) and scattering operators. The scheme is the same as for convex co-compact hyperbolic quotients \cite{PP,GRZ} and for quotients with rational cusps \cite{GuiCubo}. 

\subsection{The Poisson operator for a pure parabolic group}
The Poisson operator for a boundary problem, including the asymptotic one considered here, is the mapping
which carries the (asymptotic) boundary value to the solution of the equation in question in the interior 
which has this boundary value. For hyperbolic manifolds, the Schwartz kernel of this operator can be identified with the
Eisenstein series for the group, hence in this setting the Poisson operator is sometimes also called the Eisenstein operator.

Let $X_c=\Gamma_\infty\backslash \hh^{n+1}$ with $\Gamma_\infty$ an elementary discrete parabolic group of rank $k < n$ fixing $\infty$. 
The Poisson operator for this space, $P^c(s)$, is defined by
\begin{equation}\label{defP^c} 
P^c(s)f(x,y,z) =  \frac{2^{-\la+1}}{\Gamma(\la)}x^\ndemi \Delta_{F}^{\frac{\la}{2}}K_\la (x\sqrt{\Delta_{F}})f(y,z) , \qquad \la=s-n/2.
\end{equation}
This evidently satisfies $(\Delta_{X_c}-s(n-s))P^c(s)f=0$. 

\begin{prop}\label{eisensteincusp}
Let $\la=s-n/2$ and $\rho=x/(1+x)$. The operator 
\begin{equation} \label{stat1}
P^c(s): H^\infty(F) \to \rho^{s} H^\infty(X_c)+\rho^{n-s}  H^\infty(X_c)
\end{equation}
is a holomorphic family of bounded operators for $\{{\rm Re}(\la)\geq 0; \la\notin \nn\}$ and 
for all $\beta\in\nn_0, N\geq 0$, there is $C>0$ such that 
$||\pl_x^\beta (P^c(s)f)(x,\cdot)||_{H^{2N}(F)}\leq C x^{\ndemi-{\rm Re}(\la)}||f||_{L^2}$ for $x>1$.
If $f \in H^\infty(F)$, there exist $F^\pm\in H^\infty(X_c)$ such that $P^c(s)f= \rho^{n-s}F^- +\rho^{s}F^+$ and
\[ 
F^-|_{x=0}=f, \quad F^+|_{x=0}= 2^{-2\la}\frac{\Gamma(-\la)}{\Gamma(\la)}\Delta_F^\la f.
\] 
\end{prop}
\begin{proof}
The holomorphy in $\mbox{Re}(s) \geq n/2$, $s \notin n/2 + \nn$, follows from the holomorphy 
of Bessel function $K_\la$.

To check the estimate for $x \geq 1$, fix $\chi\in \calC^\infty(\rr^+)$ which equals $1$ in $(2,\infty)$ and vanishes in $[0,1]$.
Then
\[
(1+\Delta_{F})^N \pl_x^\beta (\chi(x\sqrt{\Delta_F})K_{\la}(x\sqrt{\Delta_F})f)=\pl^\beta(\chi K_\la)(x\sqrt{\Delta_{F}}) \Delta_{F}^{\frac{\beta}{2}}(1+\Delta_{F})^{N}f.
\]
Since $\sup_{t\in\rr_+}|t^{\pm L}\pl^\beta (\chi K_\la)(t)|<\infty$ for any $L\in\rr$, we deduce that for all $L\geq 0$
\begin{equation}\label{forx>1} 
|| \pl_x^\beta (\chi(x\sqrt{\Delta_F})P^c(s)f)||_{H^{2N}(F)}\leq C_L x^{n/2 -L} ||f||_{H^{\mathrm{Re}(\lambda) + \beta + 2N - L}}.
\end{equation}
On the other hand, observe that $2^{-\la+1}K_\la(t)/\Gamma(\la)= t^{-\la}G^-_\la(t)+t^{\la}G^+_\la(t)$ for some smooth functions 
$G^\pm_\la$ on $[0,\infty)$ with $G^-_\la(0)=1$, and hence
\[
(1-\chi(x\sqrt{\Delta_{F}}))P^c(s)= (1-\chi(x\sqrt{\Delta_{F}}))(x^{\ndemi-\la} G^-_\la(x\sqrt{\Delta_{F}})+x^{\ndemi+\la} \Delta_{F}^{\la}
G^+_\la(x\sqrt{\Delta_{F}})).
\]
It is straightforward that $(1-\chi(x\sqrt{\Delta_{F}}))G^\pm_\la(x\sqrt{\Delta_{F}})f\in H^\infty(X_c)$ and for $x>1$
\[ ||\pl_x^\beta ((x^2\Delta_F)^\la(1-\chi(x\sqrt{\Delta_F})) G^+_\la(x\sqrt{\Delta_F})f)||_{H^{2N}(F)}\leq 
C ||f||_{L^2(F)},\]
which proves \eqref{stat1} and the statement about $x>1$ by combining with \eqref{forx>1}. 
The asymptotic limits when $x\to 0$ come from the asymptotic expansion of $K_\la(t)$ at $t=0$, which gives that
$G^+_\la(0)=1$ and $G_\la^-(0)=2^{-\la}\Gamma(-\la)/\Gamma(\la)$.
\end{proof}

\begin{remark}\label{remarkevenexp}
The functions $F^\pm$ in the Proposition above have a Taylor expansion at $x=0$ with only even powers  of $x$, this is
a consequence of the  fact that the functions $G_\la^\pm(z)$ defined in the proof of this Proposition 
are smooth functions of $z^2\in[0,\infty)$.
\end{remark}

\subsection{Scattering theory on $X$}
We now proceed to define the scattering operator in the usual way.
\begin{prop}\label{poisson}
Let $X=\Gamma\backslash \hh^{n+1}$ be a geometrically finite hyperbolic manifold and let $\rho$ be a function as in Section \ref{seccontinu}. Suppose that $s\in \{{\rm Re}(s)> n/2, 
s\not\in (n/2+\nn_0), s(n-s) \notin \sigma_{\rm pp}(\Delta_{X})\}$, and fix any $f\in H^\infty(\pl\bbar{X})$. 
Then there is a unique solution $u_s\in \calC^\infty(X)$ to the equation  $(\Delta_{X}-s(n-s))u_s=0$ for which
$u_s|_{x\geq 1}\in L^2(X)$, and such that there exist functions $G_\pm \in H^\infty(X)$ with
\[
u_s= \rho^sG_++\rho^{n-s}G_-, \qquad \mbox{where} \quad G_-|_{\rho=0}=f.
\] 
\end{prop}
\begin{proof}
The problem is solved in each cusp neighbourhood $\calU^c_j$ using the Poisson operator $P^c_{j}(s):H^\infty(F_j)\to 
\calC^\infty(\Gamma_j \backslash\hh^{n+1})$ in each model space $\Gamma_j\backslash\hh^{n+1}$ given by \eqref{defP^c} . 
Fix cutoff functions $\hat{\chi}^{c}_j$, $\chi^{c}_j$ and $\chi_j^r$ as above and write $\phi^{c}_j=\chi^{c}_j|_{x=0}$; then set
\[
u^c=\sum_{j\in J^c}\hat{\chi}^c_jP^c_{j}(s)\phi^c_jf , \quad u^r=\sum_{j\in J^r} \chi_j^r x^{n-s} f.
\]
where in each $\calU^r_j$, $(x,y)$ are the coordinates induced by the half-ball \eqref{halfball}. These satisfy
\[
\begin{gathered}
(\Delta_{X_c}-s(n-s))u^c=\sum_{j\in J^c}[\Delta_X,\hat{\chi}^c_j]P^c_j(s)\phi_j^c f := q^c, \\ 
(\Delta_{X_c}-s(n-s))u^r=\sum_{j\in J^r} ([\Delta_X,\chi^r_j]f+x^2\chi^r_j \Delta_yf):=q^r.
\end{gathered} 
\]
By Proposition \ref{eisensteincusp} and equation \eqref{commutator}, $q^c\in \rho^{s+2}\calC_0^\infty(\bbar{X})+
\rho^{n-s+2}\calC_0^\infty(\bbar{X}) $, while $q^r\in \rho^2\calC_0^\infty(\bbar{X})$. Remark \ref{remarkevenexp} shows that 
the $\calC_0^\infty(\bbar{X})$ functions in the expansion of $q^c$ are smooth 
functions of $x^2$ in $\calU^c_j$, 
while in $\calU^r_j$, the function $x^{-s-2}q^r$ has an even Taylor expansion in powers of $x$ if $\chi^r_j$ are taken as functions of $(x^2,y)$ in $\calU^r_j$. Then, using the 
indicial equations \eqref{indicialeq2}, \eqref{indicialeq} in each of these neighbourhoods, we remove all terms in the expansion of 
the remainder terms $q^c, q^r$ at $\rho=0$, 
just as we did in the resolvent parametrix construction. This gives a function 
$v_s=\rho^sv^+ +\rho^{n-s}v^-$ with $v^\pm \in \calC_0^\infty(\bbar{X})$, such that   
\[
(\Delta_X-s(n-s))v_s \in  \dot{\calC}_0^\infty(\bbar{X}), \quad v^-|_{\pl\bbar{X} }=f \qquad \mbox{ and }\,\, v^{\pm} \in H^\infty(X).
\]  

Finally, set
\[
u_s:=  v_s -R(s)(\Delta_X-s(n-s))v_s;
\]
the mapping properties of $R(s)$ from Corollary \ref{propertyR(s)} and the expansion of the terms involving
$P_c(s)$ from Proposition \ref{eisensteincusp} show that this is indeed a solution to the problem.

We conclude by proving uniqueness. First note that, using the indicial equations \eqref{indicialeq} and \eqref{indicialeq2}, 
the expansion of a solution of the problem at $x=0$ is determined entirely by $G_+|_{x=0}$ and $G_-|_{x=0}$.  Therefore,
if $w$ is the difference of two such solutions, then $(\Delta_X - s(n-s))w = 0$ and $w \in L^2(X)$. Since ${\rm Re}(s)>n/2$
and $s \notin \sigma_{\mathrm{pp}}(\Delta_X)$, we have $w = 0$, which concludes the proof. 
\end{proof}

We can now define the Poisson operator for $\{s\in \cc; {\rm Re}(s)>n/2, s\notin \ndemi+\nn\}$ by 
$P_X(s)f = u_s$, where $u_s$ is the solution obtained in Proposition \ref{poisson}. With $\rho$ as in Section \ref{seccontinu}, 
and for this range of $s$, we see that
\begin{equation}\label{defpoisson} 
P_X(s):  H^\infty(\pl\bbar{X} ) \longrightarrow \rho^s H^\infty (X)+ \rho^{n-s} H^\infty(X).
\end{equation}
The next result shows that this extends to the closed half-plane:
\begin{prop}\label{holomorphy}
The operator $P_X(s)$, which is holomorphic in $\{s\in \cc; {\rm Re}(s)>n/2, s\notin \ndemi+\nn\}$, admits a meromorphic 
continuation to the entire complex plane as an operator $\calC_0^\infty(\pl\bbar{X}) \rightarrow \calC^{\infty}(X)$. Moreover,
\[
\begin{gathered}
 P_X(s): H^\infty(\pl\bbar{X})\to \rho ^sH^\infty(X)+ \rho^{n-s}H^\infty(X) , \,\, {\rm if } \,\, {\rm Re}(s)\geq n/2,
\end{gathered}
\]
\end{prop}
\begin{proof}  The existence of the meromorphic continuation will follow from a slight variant of the construction
of $P_X(s)$. 

Fix any $f\in \calC^\infty_0(\pl\bbar{X})$ and construct (using the indicial equation in each chart and Borel summation) 
a function $\Phi(s)\in \rho^{n-s}\calC_0^\infty(\bbar{X})$
which satisfies
\[ 
(\Delta_X-s(n-s))\Phi(s)\in \dot{\calC}_0^\infty(\bbar{X}), \qquad \mbox{with}\quad \left. \rho^{s-n}\Phi(s)\right|_{\rho=0}=f.
\]   
Now use the resolvent to solve away this error term. This leads to the formula
\[ 
P_X(s)f:= \Phi(s)-R_X(s)(\Delta_X-s(n-s))\Phi(s). 
\] 
The right hand side obviously continues meromorphically to $\cc$ with finite rank poles. The fact that this lies 
in  $\rho^{n-s}\calC_0^\infty(\bbar{X})+\rho^sH^\infty(X)$ when ${\rm Re}(s)\geq n/2$ follows from Corollary \ref{propertyR(s)}.
\end{proof}

\begin{lem}\label{kernelPs}
The integral kernel of $P_X(s)$ is related to the integral kernel of $R_X(s)$ by 
\[ 
P_X(s; m,b')= (2s-n)[\rho(m')^{-s}R_X(s; m,m')]|_{m'=b'} , \quad m\in X, \,\, b'\in \pl\bbar{X}.
\] 
\end{lem}
\begin{proof}
This relationship is derived almost exactly as in the convex cocompact case; we sketch it for the convenience of the reader. 
Combining Green's formula and the equation $(\Delta_X-s(n-s))R_X(s;m,m')=\delta(m-m')$, we obtain
\[
\begin{split} 
P_X(s)f(m) &= \Phi(s;m)-\lim_{\eps\to 0}\int_{x(m')\geq \eps} R_X(s;m,m')(\Delta_X-s(n-s))\Phi(s;m')\, dv_g(m')\\  
 &= \lim_{\eps\to 0}\int_{x(m')=\eps} \Big(\pl_{n'}R_X(s;m,m')\Phi(s;m')-R_X(s;m,m')\pl_{n'}\Phi(s;m')\Big)\, dv_g(m');
\end{split}
\] 
here $\pl_{n'}$ is the inner unit normal to $\{x(m')=\eps\}$ acting on the $m'$ variable and $x$ is a global defining function of $\pl\bbar{X}$ as in Section \ref{seccontinu} (and $\rho=x/(1+x)$). Note too that the integration is over 
a compact set $K$ in $\bbar{X}$ because $\Phi(s)\in \calC_0^\infty(\bbar{X})$. It is not hard to check that in terms of local
coordinates $(x,y)$, $\pl_{n'}  = x\pl_x+\alpha x^2\pl_x+\sum_i\beta_i x\pl_{y_i}$ with $\alpha,\beta_i \in \calC_0^\infty(\bbar{X})$.
Hence considering the asymptotic expansions of $R_X(s;m,m')$ and $\Phi(s;m')$ and their derivatives with respect to
$\pl_{n'}$ as $x(m')\to 0$, we obtain
\[
P_X(s)f(m)=(2s-n)\int_{\pl\bbar{X}} [x(m')^{-s}R_X(s;m,m')]|_{m'=b'}f(b')dv_{\pl\bbar{X}}(b'),
\]
as desired.
\end{proof}

Combining this Lemma with Proposition \ref{smoothness} we obtain that
\begin{equation}
\label{P(s)smooth}
P_X(s)\in \rho^{s}\calC^\infty((\bbar{X}\x \pl\bbar{X})\setminus{\rm diag}_{\pl\bbar{X}}),
\end{equation}
where ${\rm diag}_{\pl\bbar{X}}:=\{(b,b)\in\pl\bbar{X}\x\pl\bbar{X}\}$. Moreover, Corollary \ref{corRe(s)>k/2}  gives
that for all $m \in X$,
\begin{equation}
\label{P(s)L^2} 
P(s; m, \cdot )\in L^2(\pl\bbar{X}, dv_{\pl\bbar{X}}), \quad \textrm{ if }\, {\rm Re}(s)> \bar{k}/2,
 \end{equation}
where $\bar{k}$ is the maximum of the ranks of all nonmaximal rank cusps of $X$. Indeed, using the coordinates $(v,z)$ 
from \eqref{coorduv} on the boundary of a cusp neighbourhood $\calU^c_j \cap \{x=0\}$ of a rank $k$ cusp, 
the measure on $\pl\bbar{X}$ equals $|v|^{2k}dvdz$, and $P_X(s;m,v)= \calO(|v|^{-k})$ as $|v|\to 0$.

The resolvent and Poisson kernels are also related  by a functional equation.
\begin{lem}\label{fcteq}
There is an identity
\begin{equation}\label{fcteqR} 
\begin{split}
R_X(s;m,m')-R_X(n-s;m,m') = \hfill \\
\hfill \frac{1}{(2s-n)}\int_{\pl\bbar{X}}P_X(s;m,b)P_X(n-s;m',b)\, dv_{\pl\bbar{X}}(b).
\end{split}
\end{equation} 
which holds for any $m,m'\in X$ when $|{\rm Re}(s)-n/2|<1/2$. 
\end{lem} 
\begin{proof} The proof is much the same as the one of Proposition 2.1 in \cite{Gui} or  Theorem 1.3 in \cite{FHP}. 
Use the coordinates $(u,v,z)$ from \eqref{coorduv} in each cusp neighbourdhood $\calU_j^c$; thus $u$ is a boundary 
defining function of $\pl\bbar{X}$ in $\calU_j^c\cap \bbar{X}$. We extend it to a global boundary defining function, 
still denoted $u$, for $\pl\bbar{X}$ on all of $\bbar{X}$. For $\eps>0$ small, we use Green's formula as in \cite[Prop 2.1]{Gui} to get
\begin{equation}\label{green} 
\begin{gathered} 
R_X(s;m,m')-R_X(n-s;m,m')= \\
-\int_{u(b)=\eps}\Big(R_X(s;m,b)\pl_nR_X(n-s;b,m')-
\pl_nR_X(s;b,m')R_X(n-s;b,m')\Big)dv_{g}(b),
\end{gathered}
\end{equation}
where $\pl_n$ is the inner unit normal to $\{u=\eps\}$. The metric in each regular neighbourhood $\calU^r_j$ has the form 
$g=(du^2+h_0+\calO(u))/u^2$, where $h_0$ is a metric on $\pl\bbar{X}$, while in each cusp neighbourhood $\calU^c_j$ it 
appears as
\[
g=\frac{du^2+|dv|^2+(u^2+|v|^2)^2|dz|^2}{u^2}.
\]  
Thus $\pl_n$ equals $u\pl_u$ in cusp neighbourhoods and $u\pl_u+\alpha u^2\pl_u +\sum_i \beta_i u\pl_{y_i}$,
with $\alpha,\beta_i\in \calC_0^\infty(\pl\bbar{X})$, in regular neighbourhoods. 

Introduce a partition of unity to localize to these different neighbourhoods. From the Lemma \ref{kernelPs}, the structure of 
$R_X(s)$ in $\calU^r_j$ and its symmetry $R_X(s;m,b)=R_X(s;b,m)$, we obtain the contribution to the integrand from $\{u=\eps\}\cap 
\calU^r_j$ in the limit as $\eps\to 0$ is given by\footnote{Notice that $(u/x)|_{\pl\bbar{X}}=|v|^{2}$
and the terms involving extra powers of $|v|$ from writing $P_X(s),P_X(n-s)$ as weighted 
restrictions to $\pl\bbar{X}$ cancel out with the extra powers of $|v|$
coming from writing the volume measure $dv_{g}(b)$ in terms of $dv_{\pl\bbar{X}}={\rm dvol}_{h_0}$.} $(2s-n)^{-1}P_X(s;m,b)P_X(n-s;m',b)dv_{\pl\bbar{X}}(b)$.
Applying analogous arguments, using Lemma \ref{kernelPs}, Corollary \ref{corRe(s)>k/2} and dominated convergence 
(the measure restricted on $\{u=\eps\}$ is $dv_g=(\eps^2+|v|^2)^k dv dz$ and so $P_X(s;m,\cdot )\in L^2$ on this hypersurface),
we find that the contribution from the cusp neighbourhoods is exactly the same. 
\end{proof}

We can now define the scattering operator $S_X(s): H^\infty(\pl\bbar{X})\to H^\infty(\pl\bbar{X})$ for ${\rm Re}(s)\geq n/2$ 
and $s\notin n/2+\nn$ by
\begin{equation}
\label{defscat} 
S_X(s)f := G_+|_{\pl\bbar{X}} 
\end{equation}
where $G_+\in H^\infty(X)$ is the function appearing defined in Propositions \ref{poisson} and \ref{holomorphy} for the expansion of 
$P_X(s)f$ at $\pl\bbar{X}$. From Theorem \ref{mainth} and the construction of $P_X(s)f$,
the operator $S_X(s)$ has a meromorphic continuation as a continuous operator $\calC_0^\infty(\pl\bbar{X}) \to 
\calC^{\infty}(\pl\bbar{X})$ to $\cc\setminus(n/2+\nn)$ with finite rank poles. 
\begin{lem}\label{kernelS}
The Schwartz kernel of $S_X(s)$ is given by 
\[
S_X(s; b,b')=[\rho(m)^{-s}P_X(s; m,b')]|_{m=b} , \quad b,b'\in\pl\bbar{X}. 
\]
Furthermore, for any $\varphi\in \calC_0^\infty(\pl\bbar{X})$, $\varphi S_X(s)\varphi$ is a classical pseudodifferential 
operator of order $2s-n$.
\end{lem} 
\begin{proof} The first statement follows from the relationship $S_X(s)f=\lim_{x\to 0}(x^{-s}P_X(s)f)$ when ${\rm Re}(s)<n/2$ 
and the meromorphic extension. That $S_X(s)$ is pseudodifferential follows from the parametrix 
construction in the proof of Prop. \ref{smoothness} 
and from the formula of the models $R_j(s)$ in $\calU_j$ constructed in \S \ref{Guillopezworski}: indeed  by \cite[Prop 4.12]{PP},
$[(x(m)x(m'))^{-s}R_j(s,m,m')]_{(m,m')=(y,y')}= f_s(y)f_s(y')|y-y'|^{-2s}$ for some smooth functions $f_s$ in
the chart $\calU_j$ defined in the proof of Prop. \ref{smoothness}.
\end{proof}

\begin{lem}\label{functionalequations}
For ${\rm Re}(s)=n/2$ and $s\not=n/2$, there are identities
\[ 
P_X(s)=P_X(n-s)S_X(s) , \quad S_X(n-s)S_X(s)=S_X(s)S_X(n-s)={\rm Id}.
\]
\end{lem}
\begin{proof} By Lemma \ref{kernelS}, the Schwartz kernel of $S_X(s)$ in $\{{\rm Re}(s)<n/2\}$ lies in $L^1_{\rm loc}(\pl\bbar{X}\x
\pl\bbar{X})$, and from Corollary \ref{corRe(s)>k/2}, we also have $S_X(s;b,\cdot )\in L^2(\pl\bbar{X}\setminus B_\eps(b),
dv_{\pl\bbar{X}})$ for all $b\in\pl\bbar{X}$, where $B_\eps(b)$ is a ball of small radius $\eps>0$ in $\pl\bbar{X}$. Now fix 
$m\in X, b\in \pl\bbar{X}$ and ${\rm Re}(s)\in ((n-1)/2,n/2)$, and multiply \eqref{fcteqR} by $(2s-n)\rho(m')^{-s}$ and let
$m'\to b$. By \eqref{P(s)smooth}, \eqref{P(s)L^2} and the decay and regularity properties of $S(s;b,b')$ stated above, we deduce that 
\[
P_X(s;m,b)=\int_{\pl\bbar{X}} P_X(n-s; m,b')S_X(s;b,b')\, dv_{\pl\bbar{X}}(b').
\]
In particular, the integral converges. From the symmetry of the resolvent, we also have $S_X(s;b,b')=S_X(s;b',b)$,
so $P_X(s)=P_X(n-s)S_X(s)$ for ${\rm Re}(s)\in ((n-1)/2,n/2)$ as operators $\calC_0^\infty(\pl\bbar{X}) \to \calC^\infty(X)$. 
However, this extends to $|{\rm Re}(s)-n/2|\leq 1/2$ meromorphically, in view of the mapping properties of $S(s)$ and $P(s)$. 
The functional equation for $S(s)$ is an easy consequence: one has for ${\rm Re}(s)=n/2$ (and $s\not=n/2$) 
\[ 
P_X(n-s)=P_X(s)S_X(n-s)=P_X(n-s)S_X(s)S_X(n-s)
\]
as operators on $H^\infty(\pl\bbar{X})$, but $P_X(n-s)$ is injective on $H^\infty(\pl\bbar{X})$ by construction.
\end{proof}

\appendix
\section{Bessel functions}
We gather some definitions and estimates regarding Bessel functions; all of this can be found in \cite[Chap. 9]{AbSt}. 
For $\alpha\in\rr,\nu\in\cc$ and $z\in\rr^+$, the Bessel (resp. modified Bessel) 
functions $J_\alpha,H^{(1)}_\alpha$ (resp. $I_\nu, K_\nu$) are defined by    
\begin{equation}\label{besseldef}
\begin{gathered}
J_\alpha(z):=\sum_{m=0}^\infty \frac{(-1)^m}{m!\Gamma(m+\alpha+1)}(\frac{z}{2})^{2m+\alpha},
 \quad H^{(1)}_\alpha(z):=\frac{(J_{-\alpha}(z)-e^{-\alpha\pi i}J_{\alpha}(z))}{i\sin(\alpha\pi)}\\
I_{\nu}(z):= \sum_{m=0}^\infty \frac{1}{m!\Gamma(m+\nu+1)}(\frac{z}{2})^{2m+\nu},
\quad K_\nu(z):=\frac{\pi}{2}\frac{(I_{-\nu}(z)-I_{\nu}(z))}{\sin(\nu\pi)}. 
\end{gathered} 
\end{equation} 
They are independent solutions of the Bessel (resp. modified Bessel) equation on $\rr^+$, 
$z^2\pl_z^2u+z\pl_zu+(z^2-\alpha^2)u=0$ (resp. $z^2\pl_z^2u+z\pl_zu-(z^2+\nu^2)u=0$).
For $\nu\in\cc$ and $\alpha\in\rr$, we have
\begin{equation}\label{estbessel}
\begin{gathered}
|J_\alpha(z)|\leq \frac{2^{-\alpha}z^\alpha}{\Gamma(\alpha+1)}, \,\, \forall z,\alpha>0;
\quad  |J_\alpha(z)|=\calO(\frac{1}{\sqrt{z}})\textrm{ and } |H_\alpha^{(1)}(z)|=\calO(\frac{1}{\sqrt{z}}),\,\,  \textrm{as } z\to \infty ,\\ 
|I_\nu(z)| = \calO(\frac{e^{z}}{\sqrt{z}}), \textrm{ and } |K_\nu(z)|=\calO(\frac{e^{-z}}{\sqrt{z}}),\,\,\textrm{as }  z\to \infty 
\end{gathered}
\end{equation}
where the constants in each $\calO(\cdot)$ depend on $\nu,\alpha$.\\

\textbf{Ackowledgement}. C.G.\ is grateful to G.~Carron for his explanations of the Fourier decomposition of a flat $3$-dimensional  bundle. We thank L.~Guillop\'e for comments and the referees 
for their careful reading and suggestions. Part of this work was done while C.G.\ was visiting the Math. Dept. of Stanford. C.G.\ is supported by  ANR grant 09-JCJC-0099-01. R.M.\ is partially supported by NSF grant DMS-0805529.

\end{document}